\documentclass[final,oneside,onecolumn]{article}

\usepackage[body={6in,9in},top=1in,left=1in]{geometry}
\usepackage{overpic}
\usepackage{graphicx}
\usepackage{amsmath,amsthm,amssymb}
\usepackage{amsfonts}
\usepackage{paralist}
\usepackage{alltt}
\usepackage{url}
\usepackage{color}
\usepackage{algorithm,algorithmicx,algpseudocode}
\usepackage{amsmath}
\usepackage{amssymb}\usepackage{enumitem}
\usepackage{cite}
\usepackage{caption}
\usepackage{subcaption}
\usepackage{colortbl}
\usepackage{tikz,cite,enumitem}
\usepackage{longtable}
\usepackage{longtable}
\usepackage{fancyhdr}
\usepackage{gensymb}
\def \vu{ \mathbf{u} }

\def \vx{ \mathbf{x} }

\def \vq {\mathbf{q} }

\newcommand{\uc}{u^{\rm c}}
\newcommand{\us}{u^{\rm s}}
\newcommand{\ue}{v^{\rm e}}
\newcommand{\uo}{v^{\rm o}}

\newcommand{\lf}{\left}
\newcommand{\rt}{\right}

\def \vu{ \mathbf{u} }

\def \vx{ \boldsymbol{x} }

\newcommand{\disk}{s}
\newcommand{\sphere}{s}

\def \vu{ \mathbf{u} }

\def \vx{ \boldsymbol{x} }

\newcommand{\ds}{\displaystyle}
\newcommand{\Sph}{\mathbb{S}^2}

\usepackage{hyperref}

\newtheorem{remark}{Remark}
\newtheorem{definition}{Definition}
\newtheorem{lemma}{Lemma}
\newtheorem{theorem}{Theorem}

\begin{document}

\title{Barycentric interpolation formulas for the sphere and the disk}
%\title{Fast, spectrally accurate interpolation on the sphere and the disk}
\author{Michael Chiwere\footnote{Boise State University, Department of Mathematics, 1910 University Drive, Boise, ID 83725-1555 (michaelchiwere@u.boisestate.edu, gradywright@boisestate.edu).} $^{,}$\footnote{Corresponding author.} \, and Grady B. Wright$^*$}

\date{}
\maketitle

\begin{abstract}

Spherical and polar geometries arise in many important areas of computational science, including weather and climate forecasting, optics, and astrophysics. In these applications, tensor-product grids are often used to represent unknowns. However, interpolation schemes that exploit the tensor-product structure can introduce artificial boundaries at the poles in spherical coordinates and at the origin in polar coordinates, leading to numerical challenges, especially for high-order methods. In this paper, we present new bivariate trigonometric barycentric interpolation formulas for spheres and bivariate trigonometric/polynomial barycentric formulas for disks, designed to overcome these issues. These formulas are also efficient, as they only rely on a set of (precomputed) weights that depend on the grid structure and not the data itself. The formulas are based on the Double Fourier Sphere (DFS) method, which transforms the sphere into a doubly periodic domain and the disk into a domain without an artificial boundary at the origin. For standard tensor-product grids, the proposed formulas exhibit exponential convergence when approximating smooth functions. We provide numerical results to demonstrate these convergence rates and showcase an application of the spherical barycentric formulas in a semi-Lagrangian advection scheme for solving the tracer transport equation on the sphere.

\end{abstract}
	
\section{Introduction\label{sec:1}}

Spherical and polar geometries are central to many important areas of computational science, including climate modeling and weather forecasting~\cite{1980:HaltinerWilliams}, geodynamics~\cite{gerya2019introduction}, cosmology~\cite{hu2002cosmic}, material science~\cite{HAN20191}, optics~\cite{mahajan2007orthonormal}, and astrophysics~\cite{godon1997numerical}.  A key computational challenge that arises in many of these applications is interpolation, which is essential for such tasks as coupling models, filling in gaps between observations, evaluating numerical simulations at arbitrary points, and constructing surrogate models of complicated physical processes.  Tensor-product spherical or polar grids are commonly employed in these applications due to their straightforward implementation, ease for visualization, alignment with geographic coordinates (in the case of the sphere), natural handling of boundary conditions and symmetries, direct mapping to computer memory layouts, and efficiency and accuracy in performing spectral transforms on the sphere~\cite{shtns} and disk~\cite{janssen2007computing}. 

However, these grids also present challenges for interpolation, particularly for high-order accurate methods.  One major challenge is that interpolation schemes that directly exploit the tensor-product grid structure introduce artificial boundaries at the north and south poles of the spherical coordinate system and at the origin of the polar coordinate system---collectively known as the ``pole problem''~\cite[Ch.\ 18]{boyd2001chebyshev}.  The Double Fourier Sphere (DFS) method, first proposed by Merilees~\cite{merilees1973pseudospectral}, addresses this issue by transforming the sphere into a doubly periodic domain (i.e., a torus) and the disk into a domain without an artificial boundary at the origin.  This transformation makes it possible to apply high-order methods based on bivariate trigonometric polynomial approximations in latitude-longitude coordinates for the sphere~\cite{townsend2016computing} and bivariate approximations based on trigonometric and algebraic polynomials in angular-radial coordinates that are smooth over the entire disk~\cite{wilber2017computing}.  

While the DFS method forms the foundation of many high-order accurate techniques for solving a wide-range of problems  (e.g.,~\cite{boyd1978choice,yee1980studies,fornberg1995pseudospectral,orszag1974fourier,Cheong2000261,townsend2016computing,wilber2017computing,fortunato2024high,slomka2018stokes,spotz1998fast,LaytonSpotz2003,yoshimura2021improved,shen2000new}), there has been surprisingly little focus on deriving barycentric interpolation formulas from this approach.\footnote{We note that the paper~\cite{berrut2021linear} does derive a barycentric formula for the disk, but it is not based on the DFS method and treats the origin of the disk as an artificial boundary.}  This paper aims to address this gap by developing the first bivariate barycentric interpolation formulas for the sphere and disk that work directly with the sampled data on tensor-product spherical or polar grids.

Barycentric interpolation formulas have gained significant popularity over the past 20 years, particularly following the influential review by Berrut and Trefethen~\cite{berrut2004barycentric}.  While their paper focused primarily on polynomial interpolation, trigonometric analogs have existed since the earlier works of de la Vall\'ee Poussin~\cite{poussin1908convergence}, Salzer~\cite{salzer1948coefficients}, Henrici~\cite{henrici1979barycentric}, and Berrut~\cite{berrut1984baryzentrische}.  One key advantage of barycentric interpolation formulas is that they are formulated in terms of a set of weights that only depend on the locations of the sampled data, making the method straightforward to implement.  These weights are also known for many common grids (e.g., equally spaced and Chebyshev) or can be precomputed and reused for any new set of data.  This makes the formulas efficient to evaluate, requiring $O(kn)$ operations for $n$ grid points and $k$ evaluation points, compared to methods based coefficient expansions of algebraic or trigonometric polynomials which require $O(n\log n + kn)$ when the grid points are conducive to fast transforms or  $O(n^2 + kn)$ when they are not.  Additionally, barycentric polynomial interpolation is known to be numerically stable~\cite{higham2004numerical}.  While stability of the barycentric trigonometric formulas can be an issue, in most cases of practical interest one does not typically encounter issues for equispaced points~\cite{austin2017numerical}.

Our goal is to use these trigonometric and polynomial barycentric formulas with the DFS technique to derive new bivariate barycentric interpolation formulas for the sphere and disk that avoid the pole problem.  The derivation leverages the symmetry properties induced by the DFS transformation, along with the barycentric formulas developed by Berrut~\cite{berrut1984baryzentrische} for interpolating even and odd functions.  The new methods are applicable to rather general tensor-product grids, but we focus our derivation on specific grids that are commonly used in practical applications.

There are several existing techniques for approximation on the sphere and disk that also avoid the pole problem.  These include splines (e.g., \cite{schumaker,ALFELD19965}), moving least squares (e.g.,\cite{W_sphere,Levin}), and radial basis functions (e.g.,\cite{primerfull, heryudono2010radial}).  These methods are more flexible in that they can be used for both structured and unstructured grids.  However, achieving high-order accuracy with splines and moving least squares can be challenging, and there can be numerical instabilities issues with radial basis functions when taking them to higher-order (see~\cite[Ch.\ 3]{primerfull} more details).  This can be especially problematic for tensor product grids where the points are highly anisotropic near the poles.  In contrast, the interpolation formulas we derive converge faster than any polynomial order with increasing grid size for samples of smooth functions, and they avoid numerical instabilities typically associated with solving linear systems.

Finally, to highlight the advantages of the new interpolation formulas for the sphere, we apply them within a semi-Lagrangian advection (SLA) scheme to solve the tracer transport equation on the sphere.  Accurate solutions of this equation are critical, as global atmospheric flows are dominated by horizontal transport.  SLA techniques are widely used in atmospheric sciences to simulate not only tracer transport, but more complex non-linear shallow water flows~\cite{staniforth1991SL}.  A key step in SLA methods is interpolation, and while higher-order accurate methods have been developed for other aspects of these schemes~\cite{LaytonSpotz2003,yoshimura2021improved,ritchie1995implementation}, interpolation is done using low order methods (e.g., bivariate cubic polynomials).  Our application of the new barycentric formulas demonstrates that incorporating high-order methods also for interpolation can significantly improve the solution accuracy when solving the transport equation.

The remainder of the paper is structured as follows: First, we review the DFS method and the symmetries it induces in Section \ref{sec:dfs}. We then derive the bivariate trigonometric barycentric interpolation formulas for the sphere in Section \ref{sec:3}, followed by the formulas for the disk in Section \ref{sec:bary_disk}.  Both sections contain numerical experiments demonstrating the accuracy of these formulas.  We present an application of the new spherical interpolation formulas to solving the transport equation on the sphere using SLA in Section \ref{sec:application}.  Concluding remarks and a discussion of future directions are given in Section \ref{sec:conclusion}.
%Finally, we make some concluding remarks and discuss future directions in Section~\ref{}formulas for interpolating on the disk are discussed, followed by numerical examples in section \ref{sec:5} and in section \ref{sec:6} we discuss improving computation time for the barycentric interpolation using the NUFFT.
	
%=====================================================================================================	
%==================================================================================
\section{The Double Fourier Sphere (DFS) method\label{sec:dfs}}
The DFS method was first proposed by Merilees~\cite{merilees1973pseudospectral} for approximation on spheres and has since undergone several developments~\cite{boyd1978choice,yee1980studies,fornberg1995pseudospectral,orszag1974fourier,Cheong2000261,townsend2016computing}, with more recent studies on its convergence properties~\cite{mildenberger2021approximation}, generalization to $d$-dimensional manifolds~\cite{Mildenberger2023}, and extension to solving surface PDEs~\cite{fortunato2024high}.  The central  idea of the method is to transform a function on the sphere to one on a rectangular domain, while preserving the periodicity of the function in both the azimuthal and polar directions.  While this transformation results in an apparent ``doubling-up'' of the function, it opens up the possibility of using bivariate trigonometric (i.e., Fourier) expansions for approximating functions on the sphere.  These properties are why the method is called the DFS method.

The transform introduced by the DFS method can be described as follows~\cite{townsend2016computing}.  Let $(x,y,z)$ be a point on the unit sphere\footnote{Spheres of arbitrary radii can also be treated by appropriate scaling} and parameterized in spherical coordinates as
\begin{equation*}
x = \cos\phi\sin\theta, \quad y = \sin\phi\sin\theta, \quad z = \cos\theta,\qquad (\lambda,\theta)\in[0,2\pi]\times[0,\pi],
\end{equation*}
where $\phi$ and $\theta$ are the azimuthal and polar coordinates, respectively.  A function $f(x,y,z)$ on a sphere is transformed to these coordinates as
\begin{align}
f(\phi , \theta) =
f(\cos \phi \sin \theta , \sin \phi  \sin \theta , \cos \theta), \quad (\phi , \theta) \in [0 ,2\pi]\times [0, \pi].
\label{eq:f_sph}
\end{align}
However, while the transformed function is $2\pi$ periodic in $\phi$, artificial boundaries have been introduced in $\theta$ at $0$ and $\pi$, and the inherent periodicity in this coordinate has been destroyed.  To recover periodicity in $\theta$, the DFS method associates $f$ with the following ``doubled-up''  function on $[0, 2\pi]\times [-\pi ,\pi]$:
\begin{equation}
\label{eq:ftilde_sph}
		\tilde{f}(\phi , \theta) = 
		\begin{cases}
			g(\phi ,\theta), & (\phi , \theta) \in [0 , \pi]\times [0, \pi], \\[3pt]
			h(\phi-\pi, \theta), & (\phi , \theta)\in [\pi, 2\pi] \times [0,\pi], \\[3pt]
			g(\phi-\pi , -\theta), & (\phi , \theta) \in [\pi, 2\pi] \times [-\pi, 0], \\[3pt]
			h(\phi+\phi, -\theta), & (\phi , \theta) \in [0, \pi]\times [-\pi , 0],
		\end{cases}
	\end{equation}
where $g(\phi , \theta) = f(\phi  , \theta)$ and $h(\phi , \theta) = f(\phi + \pi, \theta)$ for $(\phi , \theta) \in [0,\pi] \times [0,\pi].$ The new function $\tilde{f}$ has the properties that $\tilde{f}(\phi,\theta) = f(\phi,\theta)$ for $(\phi , \theta) \in [0 ,2\pi]\times [0, \pi]$ and $\tilde{f}(\phi,\theta) =  f(\phi-\pi,-\theta)$ for $(\phi , \theta) \in [0 ,2\pi]\times [0, \pi]$.  This last relationship is referred to as a \textit{glide reflection}.  The function $\tilde{f}$ is $2\pi$-periodic in both $\phi$ and $\theta$ and constant along the lines $\theta = 0$ and $\theta =\pm\pi$, which corresponds to the north and south poles.  Figure~\ref{fig:football} gives a visual summary of the DFS method.
% This transformation, however, creates an artificial singularity at the poles since, for instance, all points with $\theta =0$ and $\phi \in \left[0, 2\pi\right]$ map to the single point at the north pole $(0,0,1)$.	
\begin{figure} 
 \centering
 \begin{minipage}{.49\textwidth} 
 \centering
   \begin{overpic}[width=.5\textwidth]{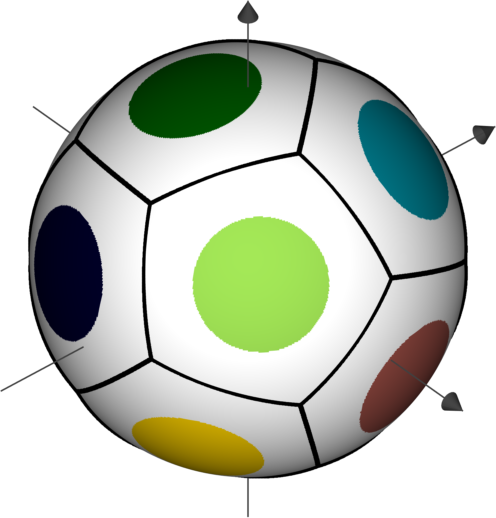}
   \end{overpic}
    \begin{overpic}[width=.95\textwidth]{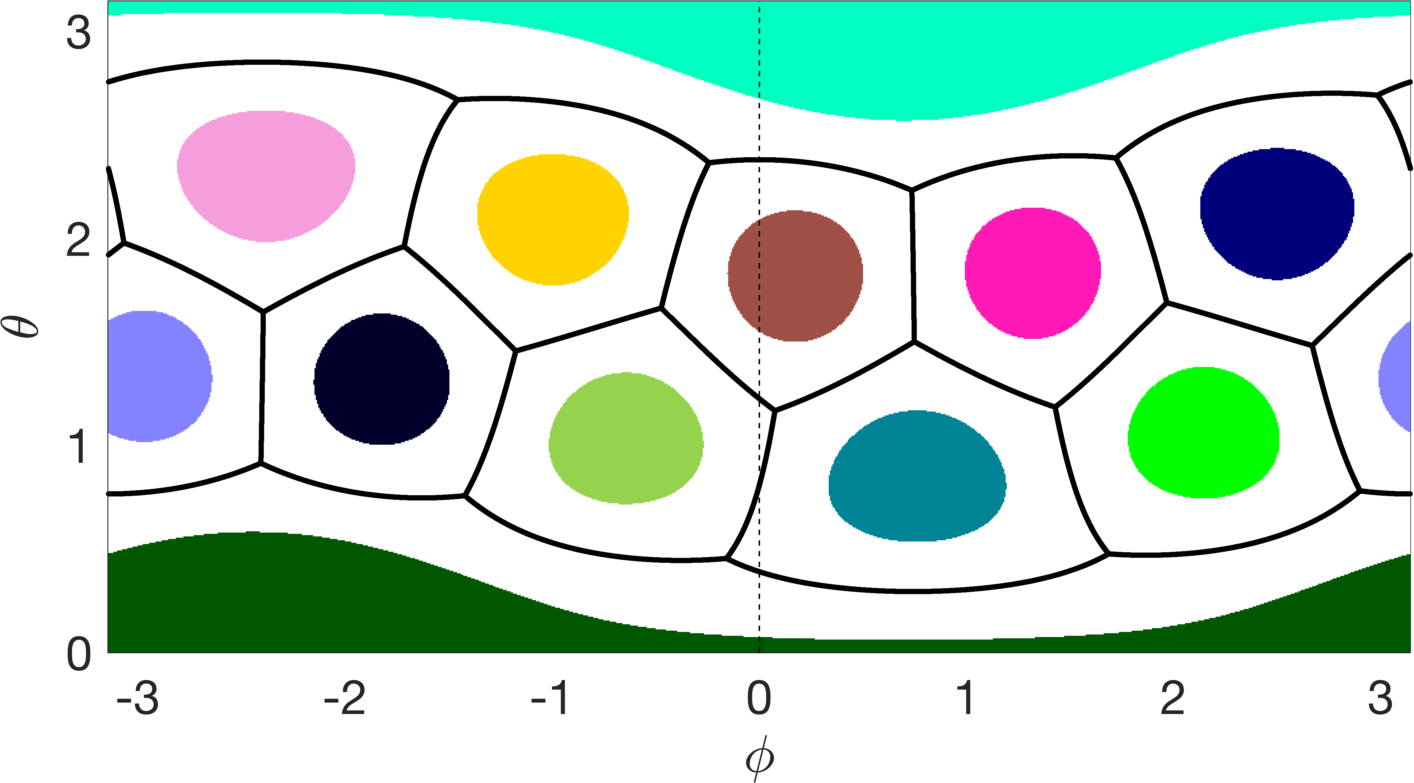}
    \put(0,105) {(a)}
    \put(0,60) {(b)}
   \end{overpic}
 \end{minipage}
\begin{minipage}{.49\textwidth} 
   \begin{overpic}[width=0.95\textwidth]{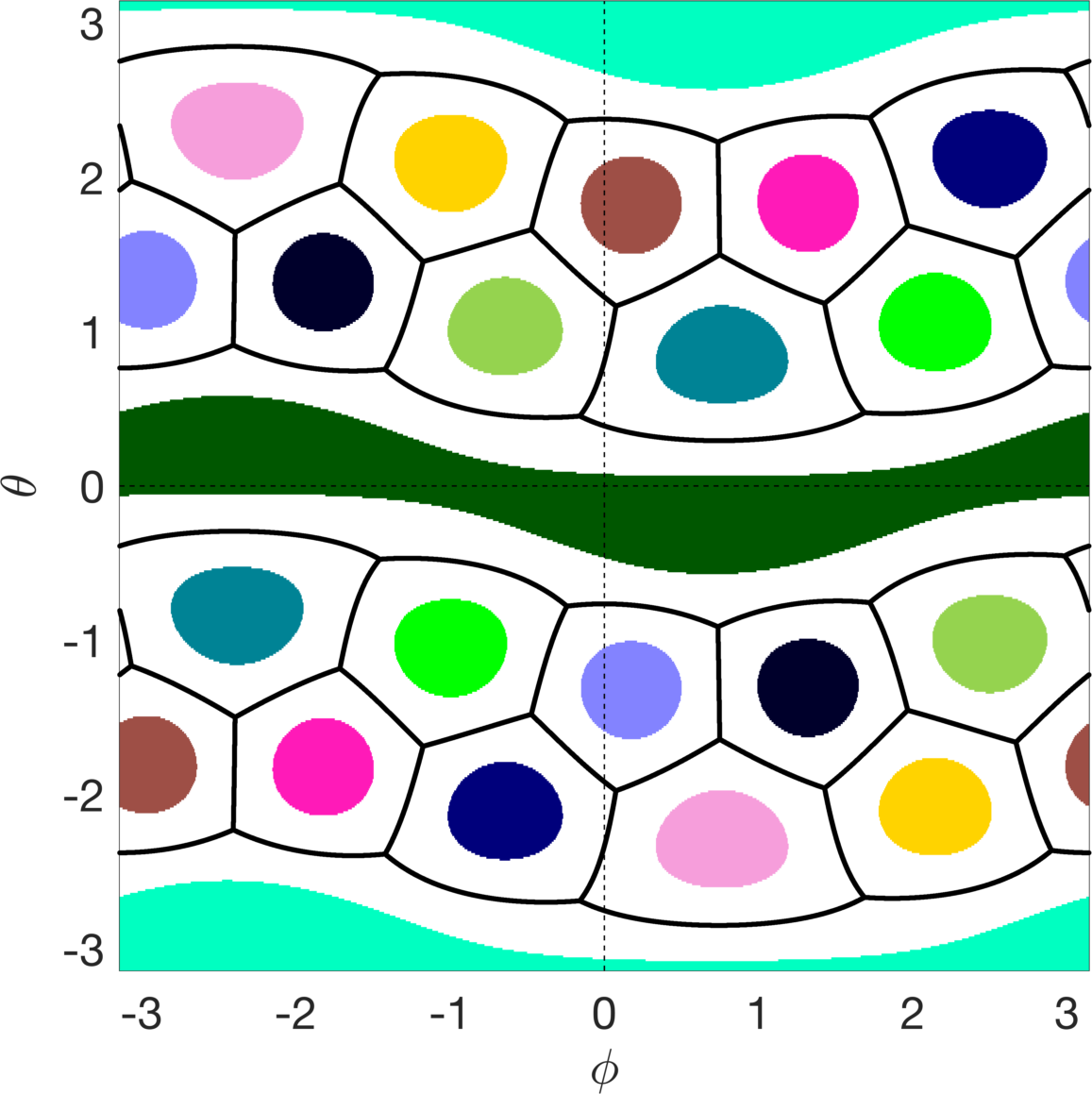}
  \put(-5,100) {(c)}
 \end{overpic}
 \end{minipage}
 \caption{The DFS method applied to a football (or soccer ball) pattern on the sphere. (a) Football pattern---note the rotation to avoid symmetries  with respect to the standard coordinate axes. (b) The projection of the football pattern using spherical coordinates. (c) Football pattern after applying the DFS method resulting in a doubly periodic pattern.}
 \label{fig:football}
\end{figure}

%=================================================================================================================
%\subsection{An analogue of DFS for the disk}
There is a DFS analog for functions defined on a disk that have similarly gone through several developments~\cite{heinrichs2004spectral,eisen1991spectral,shen2000new,fornberg1995pseudospectral,wilber2017computing}. The central idea is to transform a function on a disk to one on a rectangular domain that preserves the periodicity of the function in the polar direction and contains no artificial boundaries in the radial direction at the origin of the disk.  This can be done as follows~\cite{wilber2017computing}.  First, a function $f(x,y)$ defined in Cartesian coordinates on the unit disk\footnote{Disks of arbitrary radii can also be treated by appropriate scaling} is written in polar coordinates as
\begin{align}
f(\phi, \rho) = f(\rho\cos \phi , \rho \sin\phi) \qquad (\phi,\rho) \in [0,2\pi]\times [0,1],
\label{eq:f_disk}
\end{align}
where $\phi$ and $\rho$ are the polar and radial coordinates, respectively.  Similar to the sphere, the transformed function is $2\pi$ periodic in $\phi$, but now has an artificial boundary at $\rho=0$, which corresponds to the center of the disk.  To remove this boundary, the disk counterpart to the DFS method associates $f$ to the ``doubled-up" function $\tilde{f}$ on $[0,2\pi]\times [-1,1]$:
\begin{equation}
\label{eq:ftilde_disk}
\tilde{f}(\phi, \rho) = \begin{cases}
g(\phi ,\rho), & (\phi,\rho) \in [0 , \pi]\times [0, 1], \\[3pt]
h(\phi-\theta,\rho), & (\phi,\rho)\in [\pi, 2\pi] \times [0,1], \\[3pt]
g(\phi-\pi , -\rho), & (\phi , \rho) \in [\pi, 2\pi] \times [-1, 0], \\[3pt]
h(\phi + \pi, -\rho), & (\phi , \rho) \in [0, \pi]\times [-1 , 0],
\end{cases}
\end{equation}
where $g(\phi, \rho) = f(\phi  , \rho)$ and $h(\phi, \rho) = f(\phi+\pi, \rho)$ for $(\phi, \rho) \in [0,\pi]\times [0,1]$.  The new function $\tilde{f}$ is $2\pi$-periodic in $\phi$ and does not have a boundary at $\rho=0$.   Furthermore, it has a similar glide reflection property as the sphere and is constant along the line $\rho = 0$. 

%The disk counterpart to the DFS method also results in an apparent ``doubling-up'' of the function, but it is not suitable for  applying approximations with bivariate Fourier series because of the physical boundary at $\rho=\pm 1$. It is, however, favorable for bivariate Fourier-polynomial approximations, such as Fouirer-Chebyshev or Fourier-Legendre series.  
The disk counterpart to the DFS method also results in an apparent ``doubling-up'' of the function that is favorable for bivariate Fourier-polynomial approximations, such as Fourier-Chebyshev or Fourier-Legendre series.  In these cases, the polynomials are defined for $\rho\in[-1,1]$, rather than $[0,1]$.  For polynomial approximations based on interpolation at Chebyshev or Legendre points, this not only means  the origin is not a boundary, but also that the interpolation points are not as clustered at the origin~\cite{fornberg1995pseudospectral} (see also Figure \ref{fig:disk_pts}).  Note that even though the disk counterpart of the DFS method does not use bivariate Fourier series, it is still referred to as the DFS method for the disk~\cite{wilber2017computing}.

%In these cases, the polynomials are defined for $\rho\in[-1,1]$, rather than $[0,1]$.  For polynomial approximations based on interpolation at Chebyshev or Legendre points, this not only means  the origin is not a boundary, but also that the interpolation points are not as clustered at the origin~\cite{fornberg1995pseudospectral} (see also Figure \ref{fig:disk_pts}).  Note that even though the disk counterpart of the DFS method does not use bivariate Fourier series, it is still referred to as the DFS method for the disk~\cite{wilber2017computing}.
%======================================================================================================
\subsection{BMC symmetry and even-periodic/odd-antiperiodic decompositions\label{sec:dfs_parity}}
%The barycentric interpolation schemes introduced here are designed to enforce the symmetry associated with the functions $\tilde{f}$ in both DFS methods \eqref{eq:ftilde_sph} and \eqref{eq:ftilde_disk} so that they can always be related back to the respective domains they are defined on.  To further tease out the symmetry associated with these functions, we use a slight abuse of notation to write them as
%\begin{equation}
%\label{eq:bmc}
%\tilde{f} = 
%\begin{bmatrix}
%g & h\\ 
%\mathsf{flip}(h) & \mathsf{flip}(g)
%\end{bmatrix},
%\end{equation}
%where $\mathsf{flip}$ refers to the MATLAB command that reverses the order of the rows of a matrix.  The symmetry exhibited in \eqref{eq:bmc} is referred to as \textit{block-mirror centrosymmetric} (BMC) symmetry~\cite{townsend2016computing} and can be further classified into BMC-I and BMC-II types.  The former applies to the sphere, where $\tilde{f}$ is required to be constant when $\theta=0$ and $\theta =\pm \pi$~\cite{townsend2016computing}, while the latter applies to the disk, where $\tilde{f}$ is required to be constant when $\rho =0$~\cite{wilber2017computing}.  
The barycentric interpolation schemes introduced here are designed to enforce the symmetry associated with the functions $\tilde{f}$ in both DFS methods \eqref{eq:ftilde_sph} and \eqref{eq:ftilde_disk} so that they can always be related back to the respective domains they are defined on.  To further elucidate the symmetries associated with these functions, we use a slight abuse of notation to write them as
\begin{equation}
\label{eq:bmc}
\tilde{f} = 
\begin{bmatrix}
g & h\\ 
\mathsf{flip}(h) & \mathsf{flip}(g)
\end{bmatrix},
\end{equation}
where $\mathsf{flip}$ refers to the MATLAB command that reverses the order of the rows of a matrix.  The symmetry exhibited in \eqref{eq:bmc} is referred to as \textit{block-mirror-centrosymmetric} (BMC) symmetry~\cite{townsend2016computing} and can be further classified into BMC-I and BMC-II types if $\tilde{f}$ is the extension of a function on the sphere and disk, respectively.  
\begin{definition}(Block-mirror-centrosymmetric (BMC) functions~\cite{townsend2016computing,wilber2017computing})
A function $\tilde{f}:[0,2a]\times[-b,b]\rightarrow \mathbb{C}$, where $a,b > 0$, is a BMC function if $\tilde{f}$ is $2a$ periodic in its first variable and there are functions $g,h:[0,a]\times[0,b]\rightarrow \mathbb{C}$ such that $\tilde{f}$ satisfies~\eqref{eq:bmc}.  BMC functions can further be classified as follows:
\begin{enumerate}
\item If $\tilde{f}$ is $2b$-periodic in its second argument and constant when its second argument is equal to $0$ and $b$, then $\tilde{f}$ is a BMC-I function.%~\cite{townsend2016computing}.
\item If $\tilde{f}$ is constant when the second argument is equal to $0$, then $\tilde{f}$ is a BMC-II function.%~\cite{wilber2017computing}.
\end{enumerate}
\label{def:BMCfunction}
\end{definition}
From these definitions, we see that if $\tilde{f}$ is a BMC-I function with $a=b=\pi$, then it can be related to a function $f$ on the sphere using the functions $g$ and $h$.  This function will be at least continuous on the sphere since $\tilde{f}$ is forced to be constant at the north and south poles.  A similar result holds for the disk and BMC-II functions with $a=\pi$ and $b=1$.  However, these functions are constant only at the origin, which is required for continuity.  Our goal is to construct interpolation formulas that satisfy BMC symmetry and, in certain special cases, BMC-I or BMC-II symmetry so that they can be related back to the sphere or disk.
Another goal of our interpolation formulas is to work only on the original functions $f$ in \eqref{eq:f_sph} and \eqref{eq:f_disk}, rather than their respective extensions \eqref{eq:ftilde_sph} and \eqref{eq:ftilde_disk}. 

Both of these goals can be realized from the following decomposition (first noted in~\cite{ townsend2016computing}) of $\tilde{f}$ in \eqref{eq:bmc}: 
\begin{equation}
\label{eq:bmcI}
\tilde{f} = 
\underbrace{\frac{1}{2}\begin{bmatrix}
g-h & -(g-h)\\ \mathsf{flip}(g-h) & -\mathsf{flip}(g-h)
\end{bmatrix}}_{\tilde{f}^-} + 
\underbrace{\frac{1}{2}\begin{bmatrix}
g+h & g+h\\ \mathsf{flip}(g+h) & \mathsf{flip}(g+h)
\end{bmatrix}}_{\tilde{f}^+}.
\end{equation}
For BMC-I functions $\tilde{f}(\phi,\theta)$ with $a=b=\pi$, we see from this decomposition that  $\tilde{f}^-$ is an odd $2\pi$-periodic function in $\theta$ and $\pi-$antiperiodic\footnote{$\pi$-antiperiodic here means $\tilde{f}^{-}(\phi+\pi,\theta)=-\tilde{f}^{-}(\phi,\theta)$ for all $\phi$ and any $\theta$} in $\phi$, while $\tilde{f}^+$ is an even $2\pi$-periodic function in $\theta$, $\pi-$periodic in $\phi$, and constant when $\theta=0,\pi$.  The following Lemma uses this result to give conditions for the extension of a function $f$ defined for $[0,2\pi]\times[0,\pi$ to have BMC-I symmetry.
%\begin{equation}
%f(\phi,\theta) = f^{-}(\phi,\theta) + f^{+}(\phi,\theta),
%\end{equation}
%where $f^+(\phi , \theta) = \frac{1}{2}\left(g(\phi, \theta) + h(\phi, \theta)\right)$ is even and periodic in $\theta$ and $\pi-$periodic in $\phi$, and $f^-(\phi , \theta) = \frac{1}{2}\left(g(\phi , \theta)-h(\phi , \theta)\right)$ is odd and periodic in $\theta$ and $\pi-$antiperiodic in $\phi $.  By constructing interpolants of $f^{+}$ and $f^{-}$ with the same respective parity/periodic properties and putting these together to interpolate $f$, we will naturally enforce the BMC-I symmetry.  
\begin{lemma}\label{lemma:bmcI}
%Let $f:[0,2\pi]\times[0,\pi]\rightarrow \mathbb{C}$, $g(\phi,\theta) = f(\phi,\theta)$, and $h(\phi,\theta) = f(\phi+\pi,\theta)$ for $(\phi,\theta)\in [0,\pi]\times[0,\pi]$.  Suppose $f$ can be written as
Let $f:[0,2\pi]\times[0,\pi]\rightarrow \mathbb{C}$ have the additive decomposition
\begin{align}
f(\phi,\theta) = f^{-}(\phi,\theta) + f^{+}(\phi,\theta),
\end{align}
where $f^-(\phi , \theta) = \frac{1}{2}\left(g(\phi,\theta)-h(\phi , \theta)\right)$ and $f^+(\phi , \theta) = \frac{1}{2}\left(g(\phi, \theta) + h(\phi, \theta)\right)$, with $g$ and $h$ taking the same definitions as in \eqref{eq:ftilde_sph}.  If $f^-$ is odd, $2\pi$-periodic in $\theta$, and $\pi-$antiperiodic in $\phi$, and $f^+$ is even, $2\pi$-periodic in $\theta$, and $\pi-$periodic in $\phi$, then the extension $\tilde{f}$ of $f$ given by \eqref{eq:ftilde_sph} is a BMC function. If $f^+$  is furthermore constant for all $\phi$ when $\theta=0,\pi$, then $\tilde{f}$ is a BMC-I function.
\end{lemma}
\begin{proof}
The result follows directly by applying \eqref{eq:ftilde_sph} to $f$ and showing one gets \eqref{eq:bmcI} with the appropriate parity/periodic properties of a BMC-I function.
\end{proof}

%Another goal of our interpolation formulas is to work only of the original functions $f$ in \eqref{eq:f_sph} and \eqref{eq:f_disk}, rather than their respective extensions \eqref{eq:ftilde_sph} and \eqref{eq:ftilde_disk} in order to reduce the cost of the methods.  The following two Lemmas will aid in achieving both of these goal:
%\begin{lemma}
%Suppose $f^{+},f^{-}:[0,2\pi]\times[-\pi,\pi]\rightarrow \mathbb{C}$ and have the following properties
%\begin{enumerate} 
%\item $f^{+}$ is even and $2\pi$-periodic in $\theta$ and $\pi-$periodic in $\phi$,
%\item $f^{-}$ is odd and $2\pi$-periodic in $\theta$ and $\pi-$antiperiodic in $\phi$.  
%\end{enumerate}
%Then the extensions $\tilde{f}^{+}$ and $\tilde{f}^{-}$ of $f^{+}$ and $f^{-}$ given by \eqref{eq:ftilde_sph} are BMC-I function.
%\end{lemma}
%\begin{proof}
%\end{proof}
%If $f$ is sampled on a tensor product $(\phi,\theta)$ grid, thebivariate trigonometric 

A similar result holds for the disk, but, in this case, $\tilde{f}^-$ and $\tilde{f}^{+}$ in \eqref{eq:bmcI} only need to be odd and even in $\rho$, respectively,  but not periodic.  We summarize this in a similar lemma and omit the proof.
\begin{lemma}\label{lemma:bmcII}
Let $f:[0,2\pi]\times[0,1]\rightarrow \mathbb{C}$ have the additive decomposition
\begin{align}
f(\phi,\rho) = f^{-}(\phi,\rho) + f^{+}(\phi,\rho),
\end{align}
where $f^-(\phi , \rho) = \frac{1}{2}\left(g(\phi,\rho)-h(\phi , \rho)\right)$ and $f^+(\phi , \rho) = \frac{1}{2}\left(g(\phi, \rho) + h(\phi, \rho)\right)$, with $g$ and $h$ taking the same definitions as in \eqref{eq:ftilde_disk}.  If $f^-$ is odd in $\rho$ and $\pi$-antiperiodic in $\phi$, and $f^+$ is even in $\rho$ and $\pi-$periodic in $\phi$, then the extension $\tilde{f}$ of $f$ is a BMC function.  If $f^+$  is furthermore constant for all $\phi$ when $\rho=0$, then $\tilde{f}$ is a BMC-II function.
\end{lemma}

%By constructing interpolants of $f^{-}$ and $f^{+}$ that satisfy the parity/periodic properties of Lemmas \ref{lemma:bmcI} or \ref{lemma:bmcII} and putting these together to interpolate $f$ given in \eqref{eq:f_sph} or \eqref{eq:f_disk}, respectively, we will naturally enforce BMC symmetry.  

%So, BMC-II symmetry can naturally be enforced from interpolants of $f^{+}(\phi,\rho)$ and $f^{-}(\phi,\rho)$ that are even in $\rho$ and $\pi$-periodic in $\phi$ and odd in $\rho$ and $\pi$-antiperiodic in $\phi$, respectively.

%If the func is sampled on a latitude-longitude (rectangular) grid, we can construct a barycentric interpolant that respects BMC symmetry by constructing two bivariate barycentric interpolants: one for $f^+$ and one for $f^-$.
		
%------------------------------------------------------------------------------------
%===========================================================================================
	
\section{Barycentric interpolation formulas for the sphere\label{sec:3}}
The plan for constructing interpolants for the sphere that preserve BMC symmetry is to construct barycentric interpolation formulas that match the parity/periodic properties of $f^+$ and $f^-$ discussed in Lemma \ref{lemma:bmcI}.  These formulas can be naturally derived using a tensor product of 1D trigonometric barycentric formulas in $\phi$ and $\theta$.  Before deriving these formulas, we first review several commonly used tensor product grids to which these formulas are specifically tailored.

%order to construct the necessary formula to interpolate on the sphere, we construct bivariate trigonometric formulas for $f^+$ and $f^-$. And after that, combine them using the BMC symmetry. Since $f^-$ and $f^+$ has symmetry it is computationally efficient to construct interpolants that exploits this symmetry. We use a tensor product of trigonometric barycentric interpolants to derive formulas for $f^-$ and $f^+$.	
%---------------------------------------------------------------------------------
\subsection{Common tensor product grids}\label{sec:tensor_product}
%The formulas have been tailored for use on three commonly used tensor product latitude-longitude grids on the sphere. 
We consider the following three tensor product latitude-longitude grids on the sphere: equally spaced (EQ), shifted equally spaced (SEQ), and Gauss-Legendre (GL). These are given explicitly as follows:
%Tensor product latitude-longitude grids for the sphere almost always use equally spaced points in longitude, since then one can naturally use fast Fourier transforms to speed-up certain computations in this direction (e.g., differentiation or integration).  They then differ in how the latitude points are chosen.  We consider three of the most common choices: equally spaced (EQ), shifted equally spaced (SEQ), and Gauss-Legendre (GL). The full tensor product grids corresponding to these choices are given explicitly as follows:	
\begin{align*}
\text{EQ:}\quad (\phi_k, \theta_j)&= \left(\frac{\pi}{m}k,\frac{\pi}{n-1}j\right), \; \quad k=0,\ldots, 2m-1,\; \quad j=0,\ldots, n-1, \\
\text{SEQ:}\quad (\phi_k, \theta_j)&= \left(\frac{\pi}{m}\left(k+\frac{1}{2}\right),\frac{\pi}{n}\left(j+\frac{1}{2}\right)\right), \; \quad k=0,\ldots, 2m-1,\; \quad j=0,\ldots, n-1, \\
\text{GL:}\quad (\phi_k, \theta_j)&= \left(\frac{\pi}{m}k,\arccos (z_j)\right), \; \quad k=0,\ldots, 2m-1,\; \quad j=0,\ldots, n-1. 
\end{align*}
%\begin{align*}
%\text{EQ:}\quad (\phi_k, \theta_j)&= \left(\frac{2\pi}{m}k,\frac{\pi}{n-1}j\right), \; \quad k=0,\ldots,m-1,\; \quad j=0,\ldots, n-1, \\
%\text{SEQ:}\quad (\phi_k, \theta_j)&= \left(\frac{2\pi}{m}\left(k+\frac{1}{2}\right),\frac{\pi}{n}\left(j+\frac{1}{2}\right)\right), \; \quad k=0,\ldots, m-1,\; \quad j=0,\ldots, n-1, \\
%\text{GL:}\quad (\phi_k, \theta_j)&= \left(\frac{2\pi}{m}k,\arccos (x_j)\right), \; \quad k=0,\ldots, m-1,\; \quad j=0,\ldots, n-1. 
%\end{align*}
For the GL grid, $z_j$ are roots of the degree $n$ Legendre polynomial. See Figure \ref{fig:sphere_grids} for a visual comparison of these grids.  
\begin{figure}[ht]
\centering
\begin{tabular}{ccc}
\includegraphics[width=0.3\linewidth]{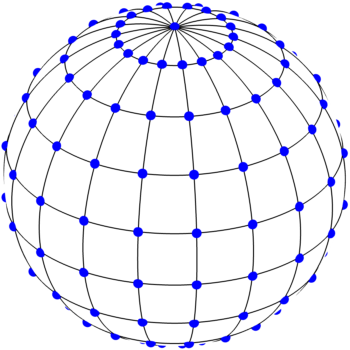} & 
\includegraphics[width=0.3\linewidth]{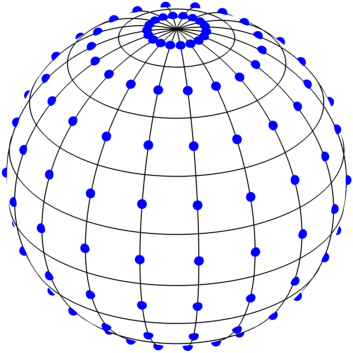} & 
\includegraphics[width=0.3\linewidth]{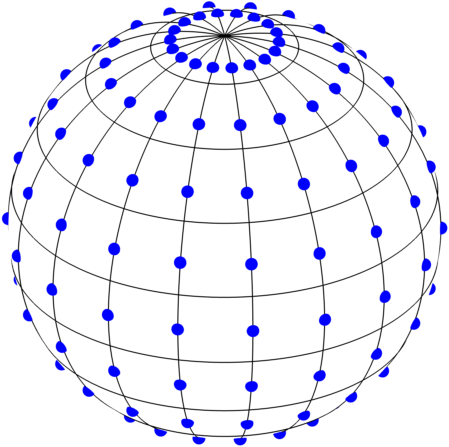} \\
EQ & SEQ & GL
\end{tabular}
\caption{Illustrations of common spherical grids for which the barycentric formulas are applicable.  For the EQ grid $m=n-1=9$, while for the SEQ and GL grids $m=n=9$.  Grid points are marked by blue dots and the solid lines correspond to an EQ grid for reference. \label{fig:sphere_grids}}
\label{fig:1}
\end{figure}

The EQ and SEQ grids are used in numerical weather prediction~\cite{spotz1998fast,LaytonSpotz2003,yee1980studies,Cheong2000261}, astrophysics~\cite{Bruegmann2013216,BartnikNorton,Tichy2006}, and even machine learning~\cite{mai2022sphere2vec}.  GL grids are most commonly used for problems that employ spherical harmonics since Gaussian quadrature rules can be used to exactly integrate  these basis functions~\cite{haiden2016evaluation,adams1999spherepack,shtns}.  

Note that tensor product latitude-longitude grids almost always use equally spaced points in longitude, since then one can naturally use the fast Fourier transform (FFT) to speed-up certain computations in this direction (e.g., differentiation or integration).  We have thus limited our focus to these cases.  Our formulas also require that there are an even number of these longitude points to preserve BMC symmetry in the sampled data.  %In problems that involve an odd number of points, one could simply use the forward and inverse FFT in longitude to interpolate the data to an even number of samples at equally spaced points before applying our barycentric formulas.

%-----------------------------------------------------------------------------------------------------------------------------
\subsection{Barycentric formulas for an even or odd periodic function}
\label{sec:even_oddt}
Let $f_j$, $j=0,\ldots,n-1$ correspond to samples of an even $2\pi$-periodic function at the set of interpolation nodes $\mathcal{S}_n=\{\theta_j\}_{j=0}^{n-1}\subseteq[0,\pi]$, where $\theta_j < \theta_{j+1}$, $j=0,\ldots,n-2$.  Then there exists a unique cosine polynomial $\uc(\theta) $ of degree $n-1$ that interpolates this data and can be written in Lagrange form as~\cite{zygmund2002trigonometric}
\begin{equation}
\label{eq:csp}
\uc(\theta ) =\sum_{j=0}^{n-1} \displaystyle{
\dfrac{\prod_{i=0,i\ne j}^{n-1}\left(\cos \theta - \cos \theta_j\right)}
{\prod_{i=0,i\ne j}^{n-1}\left(\cos \theta_j - \cos \theta_i\right)}f_j.}
\end{equation}
With the aim of reducing the computational cost of this formula, Berrut~\cite{berrut1984baryzentrische} derived the equivalent barycentric formula\footnote{In this formula the value of $u^{\rm c}(\theta_k)$, for any $k=0,\ldots,n-1$, is understood from the limiting the behavior of the formula, which gives $u^{\rm c}(\theta_k)=f_k$.  This same interpretation applies to all remaining barycentric interpolation formulas presented.}:
\begin{equation}
\label{eq:event}
\uc(\theta) = \dfrac{\displaystyle{\sum_{j=0}^{n-1}\dfrac{w_j}{\cos \theta -\cos \theta_j}f_j}}{\displaystyle{\sum_{j=0}^{n-1}\dfrac{w_j}{\cos \theta -\cos \theta_j}}}, \quad \text{where} \quad \displaystyle{w_j = \left(\prod_{\substack{i=0\\i\ne j}}^{n-1}\left(\cos \theta_j - \cos \theta_i\right)\right)^{-1}}.
\end{equation}
If the interpolation nodes $\mathcal{S}_n$ are evenly spaced, as is the case for for the EQ and SEQ grids, then \eqref{eq:event} simplifies into four possible cases that depend on whether $0 \text{ or } \pi$ are in $\mathcal{S}_n$:
\begin{equation}
\label{eq:evene}
\uc(\theta)=\dfrac{\displaystyle{\sum_{j=0}^{n-1}\dfrac{(-1)^j\delta_j\eta_j}{\cos \theta -\cos \theta_j}f_j}}{\displaystyle{\sum_{j=0}^{n-1}\dfrac{(-1)^j\delta_j\eta_j}{\cos \theta -\cos \theta_j}}}, \quad \text{where} \;
\eta_j = 
\begin{cases}
\sin \theta_j, & 0\not\in \mathcal{S}_n,  \pi \not\in \mathcal{S}_n,\\[3pt]
\cos \frac{\theta_j}{2}, &  0 \in \mathcal{S}_n ,  \pi \notin \mathcal{S}_n, \\[3pt]
\sin \frac{\theta_j}{2}, & 0\notin \mathcal{S}_n,  \pi \in \mathcal{S}_n, \\[3pt]
1, & 0\in \mathcal{S}_n, \pi \in \mathcal{S}_n,
\end{cases},
\quad\delta_j = \begin{cases}
\frac{1}{2}, & \theta_j = 0 \text{ or } \theta_j = \pi,\\
1, & \text{otherwise}.
\end{cases}
\end{equation}
	
If the samples $f_0,\ldots,f_{n-1}$ come instead from an odd $2\pi-$periodic function sampled at $\mathcal{S}_n$, then there exists a unique sine polynomial $\us(\theta)$ of degree $n-1$ that interpolates the data and can be written in Lagrange form as \cite{zygmund2002trigonometric}
\begin{equation}
\label{eq:ssp}
\us(\theta ) =\sum_{j=0}^{n-1} \displaystyle{
\dfrac{\sin \theta\prod_{i=0,i\ne j}^{n-1}\left(\cos \theta - \cos \theta_j\right)}
{\sin \theta_j\prod_{i=0,i\ne j}^{n-1}\left(\cos \theta_j - \cos \theta_i\right)}f_j.}
\end{equation}
Berrut~\cite{berrut1984baryzentrische} also derived equivalent barycentric formulas for this case that again depend on whether $0$ and $\pi$ are in $\mathcal{S}_n$:
%\begin{equation}
%\label{eq:oddt}
%s(\theta) = 
%\begin{cases}
%\csc \theta 
%\dfrac{\displaystyle{\sum_{j=0}^{n-1} \dfrac{w_j\sin \theta_j}{\cos \theta -\cos \theta_j}f_j}}{\displaystyle{\sum_{j=0}^{n-1}\dfrac{w_j}{\cos \theta- \cos \theta_j}}}, & 0\in \mathcal{S}_n , \pi \in \mathcal{S}_n,\\
%\sin \theta \dfrac{\displaystyle{\sum_{j=0}^{n-1} \dfrac{w_j\csc \theta_j}{\cos \theta -\cos \theta_j}f_j}}{\displaystyle{\sum_{j=0}^{n-1}\dfrac{w_j}{\cos \theta- \cos \theta_j}}}, & \text{otherwise},
%\end{cases}
%\end{equation}
\begin{equation}
\label{eq:oddt}
\us(\theta) = 
\begin{cases}
\sin \theta 
\dfrac{\displaystyle{\sum_{j=0}^{n-1} \dfrac{w_j\sin \theta_j}{\cos \theta -\cos \theta_j}f_j}}{\displaystyle{\sum_{j=0}^{n-1}\dfrac{w_j\sin^2\theta_j}{\cos \theta- \cos \theta_j}}}, & 0\in \mathcal{S}_n , \pi \in \mathcal{S}_n,\\
\sin \theta \dfrac{\displaystyle{\sum_{j=0}^{n-1} \dfrac{w_j\csc \theta_j}{\cos \theta -\cos \theta_j}f_j}}{\displaystyle{\sum_{j=0}^{n-1}\dfrac{w_j}{\cos \theta- \cos \theta_j}}}, & \text{otherwise},
\end{cases}
\end{equation}
where $w_j$ are given in \eqref{eq:event}. 
%We note that $c(\theta)\cdot \sin (\theta)$ is a cosine polynomial of degree $n-1$ that interpolates $f(\theta)\cdot \sin \theta$ at $0=\theta_0 <\theta_1 <\dots < \theta_{n-1}=\pi$. Thus we can adapt \eqref{eq:event} to write a barycentric formula for an odd function where the interpolation points include $0$ and $\pi$ as \cite{berrut1984baryzentrische}
%	\begin{equation}
%		\label{eq:oddt1}
%		s(\theta) =\dfrac{1}{ \sin \theta} \dfrac{\displaystyle{\sum_{j=0}^{n-1} \dfrac{w_j\sin \theta_j}{\left(\cos \theta -\cos \theta_j\right)}f_j}}{\displaystyle{\sum_{j=0}^{n-1}\dfrac{w_j}{\cos \theta- \cos \theta_j}}},
%	\end{equation}
%	where $w_j$ are given in \eqref{eq:event}. 	
Similar to the even periodic formulas, when $\mathcal{S}_n$ contains equally spaced nodes \eqref{eq:oddt} simplifies considerably to
\begin{equation}
\label{eq:odde}
\us(\theta)=\sin \theta \dfrac{\displaystyle{\sum_{j=0}^{n-1}(-1)^{j}\dfrac{\xi_j}{\cos \theta - \cos \theta_j}f_j}}{\displaystyle{\sum_{j=0}^{n}(-1)^{j}\dfrac{\xi_j \sin \theta_j}{\cos \theta - \cos \theta_j}}}, \quad \text{where} \quad
\xi_j = 
\begin{cases}
1, & 0\notin \mathcal{S}_n , \pi \notin \mathcal{S}_n,\\[3pt]
\sin \frac{\theta_j}{2},& 0 \in \mathcal{S}_n , \pi \notin \mathcal{S}_n,\\[3pt]
\cos \frac{\theta_j}{2}, & 0\notin \mathcal{S}_n, \pi\in \mathcal{S}_n,\\[3pt]
\sin \theta_j, & 0 \in \mathcal{S}_n, \pi \in \mathcal{S}_n.
\end{cases}
\end{equation}

We note that when $\mathcal{S}_n$ contains transformed GL points, $\theta_j = \arccos(z_j)$, the barycentric weights in both \eqref{eq:event} and \eqref{eq:oddt} correspond to the standard standard polynomial barycentric weights for the GL points.  These weights can be computed efficiently (and stably) from GL quadrature weights, or, for large $n$, using an explicit asymptotic formula~\cite{bogaert2014iteration}.  In our work, we use the barycentric GL weights computed in the \texttt{legpts} function of Chebfun~\cite{driscoll2014chebfun}, which uses a variation of these formulas.

	%==========================================================================================
	
\subsection{Barycentric formulas for a $\pi$-periodic or $\pi$-antiperiodic function}\label{sec:periodic_aperiodic}
We now consider the case where the grid points $\phi_k$ are equally spaced and the samples $\{f_j\}_{j=0}^{2m-1}$ come from a $\pi$-periodic or $\pi$-antiperiodic function $f(\phi)$.  In both cases, $f$ is also $2\pi$-periodic so that the standard trigonometric barycentric formula derived by Henrici~\cite{henrici1979barycentric} could be used:
	\begin{equation}
		\label{eq:bary_trig}
		p(\phi)=
		\dfrac{\displaystyle{\sum_{k=0}^{2m-1}(-1)^{k} \cot \frac{\phi-\phi_k}{2}f_k}}{\displaystyle{\sum_{k=0}^{2m-1}(-1)^{k}\cot \frac{\phi-\phi_k}{2}}}.
	\end{equation}
However, we can exploit the periodicity to reduce the computational cost of this formula, since in the $\pi$-periodic case $f_{k+m} = f_{k}$ and in the antiperiodic case $f_{k+m} = -f_{k}$, for $k=0,\ldots,m-1$.  The resulting $\pi$-periodic barycentric formula is trivially then just a rescaling and truncation of the sums in \eqref{eq:bary_trig}: 
	\begin{equation}
		\label{eq:pip}
		p^{+}(\phi) =
		\begin{cases}
			\dfrac{\displaystyle{\sum_{k=0}^{m-1} (-1)^{k} \cot \left(\phi -\phi_k\right)f_k}}{\displaystyle{\sum_{k=0}^{m-1}(-1)^{k} \cot \left(\phi -\phi_k\right)}}, & m \text{ even},\\
			\dfrac{\displaystyle{\sum_{k=0}^{m-1} (-1)^{k} \csc \left(\phi -\phi_k\right)f_k}}{\displaystyle{\sum_{k=0}^{m-1} (-1)^{k}\csc(\phi-\phi_k)}}, & m \text{ odd}.
		\end{cases}
	\end{equation}
The derivation of the barycentric formula for the $\pi$-antiperiodic interpolant requires a little more effort.  Starting with \eqref{eq:bary_trig} and exploiting the relationship $f_{k+m} = -f_{k}$, $k=0,\ldots,m-1$, we obtain
	\begin{align*}
		p^{-}(\phi) &= 
%		\dfrac{\displaystyle{\sum_{k=0}^{2m-1} (-1)^{k}\cot \left(\frac{\phi -\phi_k}{2}\right)}f_k}{\displaystyle{\sum_{k=0}^{2m-1}(-1)^{k}\cot \left(\frac{\phi -\phi_k}{2}\right)}} \\
		\dfrac{\displaystyle{\sum_{k=0}^{m-1}\left[(-1)^{k}\cot \left(\frac{\phi -\phi_k}{2}\right)f_k
				+(-1)^{m+k}\cot \left(\frac{\phi -\phi_{k+m}}{2}\right)f_{k+m}\right]}}
		{\displaystyle{\sum_{k=0}^{m-1}\left[(-1)^{k}\cot \left(\frac{\phi-\phi_k}{2}\right)
				+(-1)^{m+k}\cot \left(\frac{\phi -\phi_{k+m}}{2}\right)\right]}}\\
		&=\dfrac{\displaystyle{\sum_{k=0}^{m-1} (-1)^{k}\left[\cot \left(\frac{\phi -\phi_k}{2}\right)
				-(-1)^{m}\cot \left(\frac{\phi -\phi_k }{2}-\frac{\pi}{2}\right)\right]f_k}}{\displaystyle{\sum_{k=0}^{m-1} (-1)^{k}\left[\cot \left(\frac{\phi -\phi_k}{2}\right)
				+(-1)^{m}\cot \left(\frac{\phi -\phi_k }{2}-\frac{\pi}{2}\right)\right]}}\\
%		&= \dfrac{\displaystyle{\sum_{k=0}^{m-1}(-1)^{k}\left[\frac{1}{\tan \left(\frac{\phi -\phi_k}{2}\right)}
%				+\frac{(-1)^{m}}{\cot \left(\frac{\phi -\phi_k }{2}\right)}\right]f_k}}{\displaystyle{\sum_{k=0}^{m-1} (-1)^{k}\left[\frac{1}{\tan \left(\frac{\phi -\phi_k}{2}\right)}
%				-\frac{(-1)^{m}}{\cot \left(\frac{\phi -\phi_k }{2}\right)}\right]}} \cr
%		\\
		&= \dfrac{\displaystyle{\sum_{k=0}^{m-1}(-1)^k\left[\cot \left(\frac{\phi -\phi_k}{2}\right) +(-1)^m\tan \left(\frac{\phi -\phi_k}{2}\right)\right]f_k}}
		{\displaystyle{\sum_{k=0}^{m-1}(-1)^k \left[\cot \left(\frac{\phi -\phi_k}{2}\right)- (-1)^m\tan \left(\frac{\phi -\phi_k}{2}\right)\right]}},
	\end{align*}
where in the last equality we have used $\cot \left( \alpha -\frac{\pi}{2}\right) = -\tan \left(\alpha\right)$. By further using the identities $\cot \alpha + \tan \alpha = 2/\sin (2\alpha)$ and $\cot \alpha - \tan \alpha = 2\cot (2 \alpha)$, we obtain the simplified formula
	\begin{equation}
		\label{eq:anti}
		p^{-}(\phi)=
		\begin{cases}
			\dfrac{\displaystyle{\sum_{k=0}^{m-1}(-1)^{k} \csc \left(\phi -\phi_k\right)f_k} }{
				\displaystyle{	\sum_{k=0}^{m-1} (-1)^{k}\cot \left(\phi -\phi_k\right)}}, & m \text{ even}, \\
			\dfrac{\displaystyle{\sum_{k=0}^{m-1} (-1)^{k}\cot \left(\phi -\phi_k\right)f_k}}{\displaystyle{\sum_{k=0}^{m-1} (-1)^{k} \csc \left(\phi -\phi_k\right)}}, & m \text{ odd}.
		\end{cases}
	\end{equation}
	
	%=========================================================================================
	
	\subsection{Barycentric formulas for the sphere}
Let $f_{j,k}$ denote samples of a function $f(\phi,\theta)$ on one of the grids $(\phi_k,\theta_j)$, $j=0,\ldots,n-1$, $k=0,\ldots,2m-1$, discussed in Section \ref{sec:tensor_product}~\footnote{Note that we denote the samples of $f$ as $f_{j,k}$ instead of $f_{k,j}$ to match the standard array indexing one uses for 2D grids, where the first index corresponds to the ``vertical'' coordinate $\theta$ and the second index corresponds to the ``horizontal'' coordinate $\phi$}. Then the samples of $f^-(\phi,\theta)$ and $f^+(\phi,\theta)$ from Lemma \ref{lemma:bmcI} are given as 
\begin{align}
f_{j,k}^{-} = \frac12(f_{j,k}-f_{j,k+m}) \;\text{and}\; f_{j,k}^{+} = \frac12(f_{j,k}+f_{j,k+m}),\; j=0,\ldots,n-1,\; k=0,\ldots,m-1.
\label{eq:fplus_fminus}
\end{align}  To interpolate these sets of data we use bivariate trigonometric interpolants built from tensor products of the 1D formulas from the previous two sections.  For the $f^{-}_{jk}$ data, we combine the formulas for the $\pi$-anti-periodic interpolant in $\phi$ and odd periodic interpolant in $\theta$, while for the $f^{+}_{jk}$ data, we combine the formulas for the $\pi$-periodic interpolant in $\phi$ and even periodic interpolant in $\theta$.  These formulas can be written respectively as
\begin{equation}
	\begin{aligned}
	\label{eq:fplusminus}
	\sphere^-(\phi, \theta) &=
		\dfrac{\displaystyle{\sum_{k=0}^{m-1} (-1)^k \csc \left(\phi - \phi_k\right)\us_{k}(\theta)}}{\displaystyle{\sum_{k=0}^{m-1}(-1)^k \cot \left(\phi - \phi_k\right)}}
\;\text{\&}\;
	\sphere^+(\phi, \theta) =
		\dfrac{\displaystyle{\sum_{k=0}^{m-1} (-1)^k \cot \left(\phi -\phi_k\right)\uc_{k}(\theta)}}{\displaystyle{\sum_{k=0}^{m-1}(-1)^k \cot \left(\phi -\phi_k\right)}},\;  \text{ if $m$ is even,} \\
	\sphere^-(\phi, \theta) &=
		\dfrac{\displaystyle{\sum_{k=0}^{m-1} (-1)^k \cot \left(\phi - \phi_k\right)\us_{k}(\theta)}}{\displaystyle{\sum_{k=0}^{m-1} (-1)^k \csc \left(\phi - \phi_k\right)}}
\;\text{\&}\;
	\sphere^+(\phi, \theta) =
		\dfrac{\displaystyle{\sum_{k=0}^{m-1} (-1)^k  \csc \left(\phi -\phi_k\right)\uc_{k}(\theta)}}{\displaystyle{\sum_{k=0}^{m-1} (-1)^k\csc(\phi-\phi_k)}},\;  \text{ if $m$ is odd} 
	\end{aligned}
\end{equation}
%\begin{equation}
%	\label{eq:fminus}
%	S^-(\phi, \theta) =
%	\begin{cases}
%		\dfrac{\displaystyle{\sum_{k=0}^{m-1} (-1)^k \csc \left(\phi - \phi_k\right)s_{k}(\theta)}}{\displaystyle{\sum_{k=0}^{m-1}(-1)^k \cot \left(\phi - \phi_k\right)}},\; m \text{ even,}\\
%		\dfrac{\displaystyle{\sum_{k=0}^{m-1} (-1)^k \cot \left(\phi - \phi_k\right)s_{k}(\theta)}}{\displaystyle{\sum_{k=0}^{m-1} (-1)^k \csc \left(\phi - \phi_k\right)}},\; m \text{ odd}
%	\end{cases}
%\;\text{and}\;
%	S^+(\phi, \theta) =
%	\begin{cases}
%		\dfrac{\displaystyle{\sum_{k=0}^{m-1} (-1)^k \cot \left(\phi -\phi_k\right)c_{k}(\theta)}}{\displaystyle{\sum_{k=0}^{m-1}(-1)^k \cot \left(\phi -\phi_k\right)}},\;  m \text{ even}, \\
%		\dfrac{\displaystyle{\sum_{k=0}^{m-1} (-1)^k  \csc \left(\phi -\phi_k\right)c_{k}(\theta)}}{\displaystyle{\sum_{k=0}^{m-1} (-1)^k\csc(\phi-\phi_k)}},\;  {m} \text{ odd},
%	\end{cases}
%\end{equation}
where the choice for $\us_k(\theta)$ and $\uc_k(\theta)$ in these formulas depends on the grid type, as summarized in Table \ref{tbl:lat_formulas}.  

\begin{table}[htb]
\centering
\begin{tabular}{|c||c|c|}
\hline 
Grid & $u^{\rm c}_k(\theta)$ & $u^{\rm s}_k(\theta)$ \\ 
\hline 
\hline	
EQ & $\dfrac{\displaystyle{\sideset{}{_{\prime}'}\sum_{j=0}^{n-1}\dfrac{(-1)^j}{\cos \theta -\cos \theta_j}f_{j,k}^+}}{\displaystyle{\sideset{}{_{\prime}'}\sum_{j=0}^{n-1}\dfrac{(-1)^j}{\cos \theta -\cos \theta_j}}}$ & 
$\sin \theta \dfrac{\displaystyle{\sum_{j=0}^{n-1}\dfrac{(-1)^{j}\sin \theta_j}{\cos \theta - \cos \theta_j}f_{j,k}^-}}{\displaystyle{\sum_{j=0}^{n-1}\dfrac{ (-1)^{j}\sin^2 \theta_j}{\cos \theta - \cos \theta_j}}}$ \\
\hline
SEQ &  $\dfrac{\displaystyle{\sum_{j=0}^{n-1}\dfrac{(-1)^j\sin \theta_j}{\cos \theta -\cos \theta_j}f_{j,k}^+}}{\displaystyle{\sum_{j=0}^{n-1}\dfrac{(-1)^j\sin \theta_j}{\cos \theta -\cos \theta_j}}}$ & 
$\sin \theta \dfrac{\displaystyle{\sum_{j=0}^{n-1}\dfrac{(-1)^{j}}{\cos \theta - \cos \theta_j}f_{j,k}^-}}{\displaystyle{\sum_{j=0}^{n-1}\dfrac{(-1)^{j} \sin \theta_j}{\cos \theta - \cos \theta_j}}}$ \\
\hline
GL & $ \dfrac{\displaystyle{\sum_{j=0}^{n-1}\dfrac{w_j}{\cos \theta -\cos \theta_j}f_{j,k}^+}}{\displaystyle{\sum_{j=0}^{n-1}\dfrac{w_j}{\cos \theta -\cos \theta_j}}}$ &
$\sin \theta \dfrac{\displaystyle{\sum_{j=0}^{n-1} \dfrac{w_j\csc \theta_j}{\cos \theta -\cos \theta_j}f_{j,k}^-}}{\displaystyle{\sum_{j=0}^{n-1}\dfrac{w_j}{\cos \theta- \cos \theta_j}}}$
\\
\hline
\end{tabular}
\caption{Barycentric trigonometric interpolants in latitude to the data $f_{j,k}^{+}$ and $f_{j,k}^{-}$ given in \eqref{eq:fplus_fminus} for the various grids under consideration. The lower and upper prime on a summation sign means the first and last terms are halved, and the formulas for $w_j$ are given in \eqref{eq:event}.\label{tbl:lat_formulas}}
\end{table}

%\begin{table}[htb]
%\centering
%\begin{tabular}{|c||c|c|}
%\hline 
%% & \multicolumn{2}{c|}{Formulas} \\ 
%% \hline
%% \hline
%  & $c_k(\theta)$ & $s_k(\theta)$ \\ 
%\hline
%Formula & $ c_k(\theta) = \dfrac{\displaystyle{\sum_{j=0}^{n-1}\dfrac{w_j}{\cos \theta -\cos \theta_j}f_{j,k}^+}}{\displaystyle{\sum_{j=0}^{n-1}\dfrac{w_j}{\cos \theta -\cos \theta_j}}} $ &
%$s_k(\theta) = \sin \theta \dfrac{\displaystyle{\sum_{j=0}^{n-1} \dfrac{w_j\csc \theta_j}{\cos \theta -\cos \theta_j}f_{j,k}^-}}{\displaystyle{\sum_{j=0}^{n-1}\dfrac{w_j}{\cos \theta- \cos \theta_j}}}$ \\
%\hline 
%\hline
%Grid & \multicolumn{2}{c|}{Weights} \\ 
%\hline	
%EQ & $\ds w_j = \begin{cases} \frac{(-1)^j}{2}, & j = 0\; \text{or}\; n \\ (-1)^j, & \text{otherwise} \end{cases}$ & $\ds w_j = (-1)^j \sin^2\theta_j$  \\
%\hline
%SEQ &  \multicolumn{2}{c|}{$w_j = (-1)^j\sin\theta_j$} \\
%\hline
%GL & \multicolumn{2}{c|}{$w_j =$ $j^{\rm th}$ barycentric weight for the $n$ GL points} \\
%\hline
%\end{tabular}
%\caption{Barycentric trigonometric interpolants in latitude to the data $f_{j,k}^{+}$ and $f_{j,k}^{-}$ given in \eqref{eq:fplus_fminus} for the various grids under consideration. The lower and upper prime on a summation sign means the first and last terms are halved, and the formulas for $w_j$ are given in \eqref{eq:event}.\label{tbl:lat_formulas2}}
%\end{table}

To construct a barycentric trigonometric interpolant of the data $f_{jk}$, we combine the interpolants for $\sphere^{-}$ and $\sphere^{+}$ as in \eqref{eq:bmcI}, $\sphere(\phi,\theta) = \sphere^+(\phi,\theta)+\sphere^-(\phi , \theta)$, and simplify to arrive at 
\begin{align}
	\label{eq:bary_sphere}
	\sphere(\phi , \theta) = \begin{cases}
		\dfrac{\displaystyle{\sum_{k=0}^{m-1} (-1)^k \left[\cot \left(\phi - \phi_k\right)\uc_k(\theta) + \csc \left(\phi - \phi_k\right)\us_k(\theta)\right]}}{\displaystyle{\sum_{k=0}^{m-1} (-1)^k\cot \left(\phi - \phi_k\right)}}, & m \text{ even,}\\
		\dfrac{\displaystyle{\sum_{k=0}^{m-1}(-1)^k \left[\csc\left(\phi - \phi_k\right)\uc_k(\theta) + \cot \left(\phi - \phi_k\right)\us_k(\theta)\right]}}
		{\displaystyle{\sum_{k=0}^{m-1}(-1)^k \csc \left(\phi - \phi_k\right)}}, & m \text{ odd. }
	\end{cases}
\end{align}
\begin{theorem}\label{thm:bmc_sphere}
The interpolant $\sphere:[0,2\pi]\times[0,\pi]\rightarrow\mathbb{C}$ to the data $f_{jk}$, $j=0,\ldots,n-1$, $k=0,\ldots,2m-1$, sampled from a continuous function $f$ on the sphere at the EQ, SEQ, or GL grids, is a BMC function.  Furthermore, in the case of the EQ grid, $\sphere$ is a BMC-I function.
\end{theorem}
\begin{proof}
The first result follows immediately from Lemma \ref{lemma:bmcI}.  The second result follows by noting that since $f$ is continuous on the sphere, then on the EQ grid, $f_{0,k} = \beta$ and $f_{n-1,k} = \gamma$, $k=0,\ldots,2m-1$, for some $\beta,\gamma\in\mathbb{C}$.  So, $f_{0,k}^{+}=\beta$, $f_{n-1,k}^{+} = \gamma$, and $f_{0,k}^{-}=f_{n-1,k}^{-}=0$, $k=0,\ldots,m$.  Since $c_k(\theta)$ are interpolants, we necessarily have that $c_k(\theta_0) = c_k(0)=\beta$ and $c_k(\theta_{n-1}) = c_k(\pi)=\gamma$, $k=0,\ldots,m$.  Now, $\sphere^{+}$ is constructed from $p^{+}(\phi)$ in \eqref{eq:pip} and $p^{+}$ is exact for constants, thus $\sphere^{+}(\phi,0)=\beta$ and $\sphere^{+}(\phi,\pi)=\gamma$ all $\phi$.  Finally, $\sphere^{-}(\phi,0) = \sphere^{-}(\phi,\pi) = 0$ for all $\phi$ since all $s_k(\theta)$ are odd $2\pi$-periodic functions.
\end{proof}

\begin{remark}
The interpolant \eqref{eq:bary_sphere} is also more generally applicable to other tensor product grids that use equally spaced points in longitude but different points in latitude than EQ, SEQ, or GL.  The formulas in this more general case are based on $c_k$ given in the third row of Table \ref{tbl:lat_formulas} and $s_k$ given in \eqref{eq:oddt}, with the weights $w_j$ computed according to the given points in latitude.  If these latitude points contain $0$ and $\pi$, then the interpolant will be a BMC-I function.
\end{remark}

For the EQ and SEQ grids, one can also obtain a bivariate trigonometric interpolant of a continuous function $f$ on the sphere from the bivariate discrete Fourier transform (equivalently the FFT) of the samples of $f$'s DFS extension (and a doubling of the grids to cover $-\pi < \theta < 0$).  From the uniqueness of the FFT~\cite{wright2015extension} and the barycentric interpolant \eqref{eq:bary_sphere}, these two must be equivalent.  The approximation properties of \eqref{eq:bary_sphere} for the EQ and SEQ grids then follow directly from~\cite{mildenberger2021approximation}.  In particular, for interpolating Hölder continuous functions on the sphere one obtains algebraic rates of convergence as $n$ and $m$ increase, while for infinitely smooth functions, one obtains convergence rates faster than any algebraic rates.  

Note that for the SEQ and GL grids that do not contain the poles ($\theta=0$ and $\theta=\pi$) as interpolation points, the interpolant \eqref{eq:bary_sphere} is not guaranteed to be single-valued at the poles.  While in practice this does not appear to degrade the convergence rates of these interpolants proved in~\cite{mildenberger2021approximation}, as illustrated in the next section, it could lead to issues when trying to differentiate the interpolated values in the vicinity of the poles.  
%In practice, when evaluating \eqref{eq:bary_sphere} at $\theta=0$, one could set the value to the average of all $c_k(0)$, and similarly when $\theta=\pi$, set the value of the average of $c_k(\pi)$.   This does not necessarily fix any issues with derivatives of the interpolant at the poles, but does give a single-valued result for the interpolant there.

We conclude by noting that a benefit of using the barycentric interpolant \eqref{eq:bary_sphere} over the equivalent interpolant written in terms of a bivariate Fourier series (cf.~\cite{townsend2016computing}) is that it does not require determining the Fourier coefficients.  For the EQ and SEQ grids, this amounts to saving the computation of a bivariate FFT (and complex arithmetic, as used in most standard FFT algorithms).  However, for the GL grid (or other grids that use non-equally spaced points in latitude), one would need to resort to a non-uniform FFT (NUFFT) (e.g.~\cite{ruiz2018nonuniform}), which may further increase the cost of the bivariate Fourier series approach.

\subsection{Numerical example}
To demonstrate that the barycentric interpolation formulas \eqref{eq:bary_sphere} work on the sphere, we consider using them to interpolate the function
\begin{equation}
\label{eq:funs}
%f(\phi,\theta) = \cosh \left(\sin\left(\cos \phi\sin\theta + 50 \cos\phi\sin\theta\sin \phi\sin\theta\cos\theta\right)\right);
f(\phi,\theta) = \cos \left(1 + 8\pi(\cos\phi + \sin\phi)\sin\theta + 5\sin(3\pi\cos\theta)\right);
%s = @(la,th) cos(1 + 8*pi*(cos(la).*sin(th)+sin(la).*sin(th)) + 5*sin(3*pi*cos(th)));
\end{equation}
see Figure \ref{fig:errs_sphere} (a) for a visualization of this function.  
\begin{figure}[ht]
	\centering
	\begin{tabular}{cc}
	\includegraphics[width=0.4\linewidth]{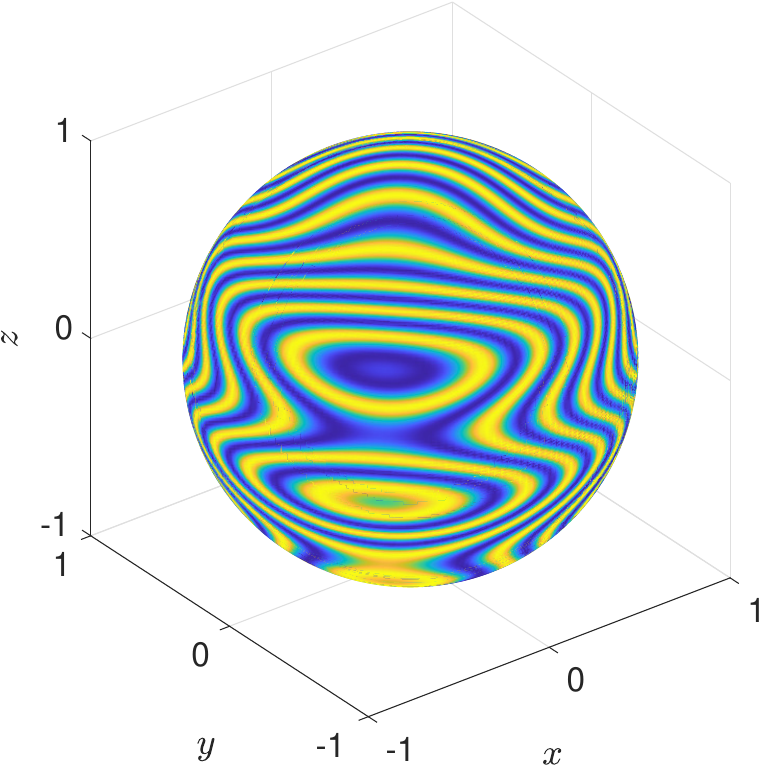} &
	\includegraphics[width=0.51\linewidth]{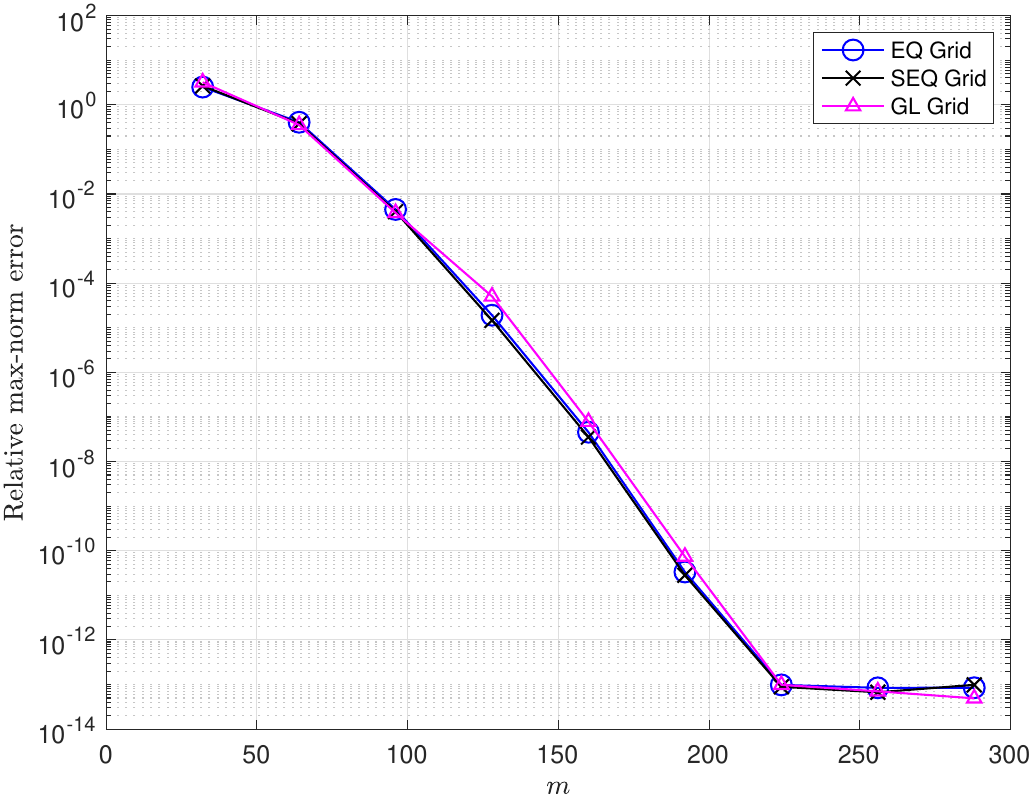} \\
	(a) & (b)
	\end{tabular}
	\caption{(a) Test function \eqref{eq:funs} on the sphere, where dark blue corresponds to -1 and bright yellow to 1. (b) Relative max-norm error in the barycentric interpolants \eqref{eq:bary_sphere} of the function in (a) for the different grids using $n = m$.\label{fig:errs_sphere}}
\end{figure}
For the numerical experiment, we select $n = m$, which gives $N=2n^2$ total grid points and samples of $f$, and evaluate the barycentric formulas \eqref{eq:bary_sphere} at a dense set of scattered points on the sphere where we compute the relative max-norm error between $\sphere$ and $f$.  Figure \ref{fig:errs_sphere} (b) displays these errors as a function of $m$ for each of the three grids under consideration.  We see the max-norm errors are similar for each of the three different grid types and that the errors decrease exponentially fast with $m$, which is expected from~\cite{mildenberger2021approximation}, since $f$ is infinitely smooth.
	
\section{Barycentric interpolation formulas for the disk}
\label{sec:bary_disk}
Similar to the sphere, the plan is to construct barycentric interpolation formulas for the disk that match the parity/periodic properties of $f^{+}$ and $f^{-}$ discussed in Lemma \ref{lemma:bmcII}.  These will again be constructed from tensor product formulas, but now consisting of 1D trigonometric barycentric formulas in $\phi$ and 1D polynomial barycentric formulas in $\rho$. The formulas for $\phi$ are identical to those derived in Section \ref{sec:periodic_aperiodic}, so only the 1D polynomials formulas need to be derived here.  Prior to this, we first review some commonly used tensor product girds that the barycentric formulas for the disk are specifically tailored.

\subsection{Common tensor product grids\label{sec:tensor_product_disk}}
We will again consider three commonly used grid types on the disk. In all the grids, the points are equally distributed in the angular direction, and we will name them depending on the nature of points in the radial direction. These are Chebyshev points of the first (CH1) and second (CH2) kinds, and Gauss-Legendre (GL) points~\cite{trefethen2019approximation}.  For given positive integers $m$ and $n$, we define these grids as
\begin{align*}
	\text{CH1:} \quad \left(\phi_k,\rho_j\right)&= \left(\frac{\pi}{m}k,\cos\left(\frac{j+\frac12}{\ell+1}\pi\right)\right), \quad k=0,\dots , 2m-1,\quad j=0,\dots, n, \\[3pt]
	\text{CH2:} \quad \left(\phi_k,\rho_j\right)&= \left(\frac{\pi}{m}k,\cos\left(\frac{j}{\ell}\pi\right)\right), \quad k=0,\dots , 2m-1,\quad j=0,\dots , n, \\[3pt]
	\text{GL:} \quad \left(\phi_k,\rho_j\right)&= \left(\frac{\pi}{m}k,z_j\right), \quad k=0,\dots , 2m-1, \quad j=0,\dots , n,
\end{align*}
where $\ell=2n$ if the origin is to be included ($\rho_{n}=0$) in the grid or $\ell=2n+1$ if it is to be omitted.  Also, $z_j$ are the $n+1$ roots in $[0,1]$ of the degree $\ell+1$ Legendre polynomial $P_{\ell+1}$ defined over $[-1,1]$, and ordered so that $z_{j} > z_{j+1}$, $j=0,\ldots,n$.  See~\cite{boyd2011comparing} for a thorough review of applications that use these types of grids.
\begin{figure}[ht]
	\centering
	\begin{tabular}{cc}
	\includegraphics[width=0.3\linewidth]{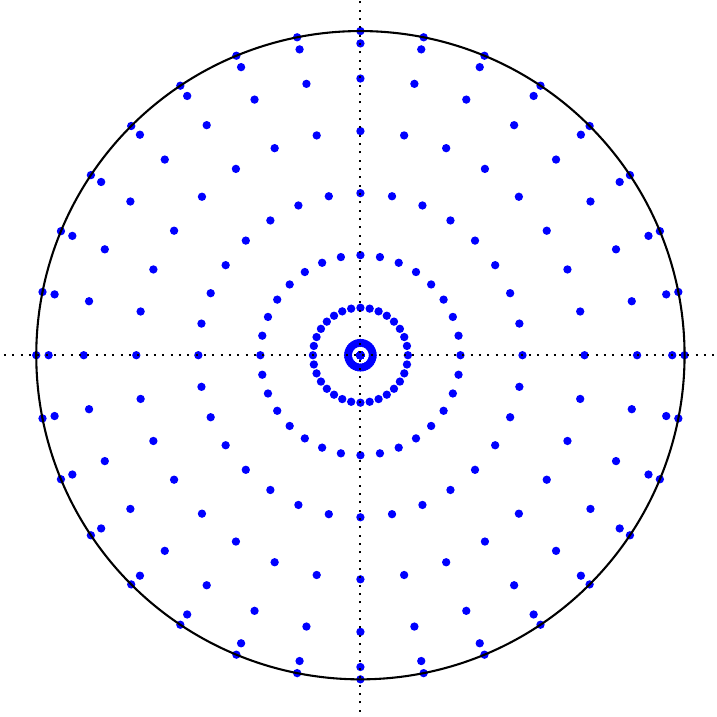} &
	\includegraphics[width=0.3\linewidth]{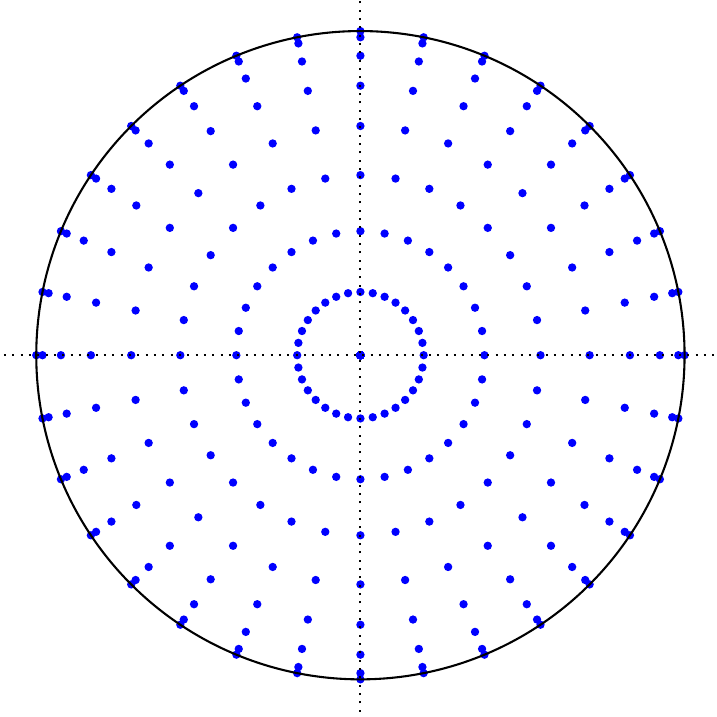} \\
	(a) & (b)
	\end{tabular}
	\caption{Comparison of Fouier-Chebyshev grids, where the points in the radial direction are (a) the Chebyshev points (of the second kind) defined over $[0,1]$ and (b) the Chebyshev points (of the second kind) defined over $[-1,1]$ and restricted to $[0,1]$ as in the CH2 definition.  Both grids have the same number of points in the polar and radial directions.\label{fig:disk_pts}}
\end{figure}

In all cases, the points in the radial direction $\rho$ are the $\ell+1$ roots or extrema in $[0,1]$ of orthogonal polynomials defined over $[-1,1]$.  As discussed in Section \ref{sec:dfs}, when the DFS method is applied to functions sampled at these points, the full set of points over $[-1,1]$ is recovered, allowing polynomials in $\rho$ defined over these points to be used.  As shown in Figure \ref{fig:disk_pts}, this approach results in less clustering around the origin compared to using radial points based on roots/extrema of orthogonal polynomials defined on $[0,1]$, as done in~\cite{berrut2021linear}.  Furthermore, as we demonstrate in Section \ref{sec:bary_formula_disk}, this approach allows for the construction of barycentric interpolants on the disk that naturally preserve BMC symmetry.

%\begin{figure}[ht]
%	\centering
%	\begin{subfigure}[b]{0.30\linewidth}
%		\centering
%		\includegraphics[width=\linewidth]{plots/CH1}
%		\caption{CH1}
%		\label{fig:4a}
%	\end{subfigure}\hfill    
%	\begin{subfigure}[b]{0.3\linewidth}
%		\centering
%		\includegraphics[width=\linewidth]{plots/CH2}
%		\caption{CH2}
%		\label{fig:4b}
%	\end{subfigure}    \hfill
%	\begin{subfigure}[b]{0.3\linewidth}
%		\centering
%		\includegraphics[width=\linewidth]{plots/legd}
%		\caption{GL}
%		\label{fig:4c}
%	\end{subfigure}\hfill
%	\caption{Arrangements of CH1, CH2 and GL grids for $m=20$.}
%	\label{fig:4.1}
%\end{figure}

%-------------------------------------------------------------------------------------------------------
\subsection{Barycentric formulas for an even or odd function}
\label{sec:even_odd}
Given samples $f_j$, $j=0,\ldots,\ell$, of a function $f$ at the the interpolation points $\rho_j$, $j=0,\ldots,n$, the unique polynomial of degree $\ell$ that interpolates the data can be written in barycentric form as~\cite{berrut2004barycentric}
\begin{align}
	v(\rho) = \dfrac{\displaystyle{\sum_{j=0}^{\ell} \frac{w_j}{\rho-\rho_j}f_j}}{\displaystyle{\sum_{j=0}^\ell\frac{w_j}{\rho-\rho_j}}}, \quad w_j =  \left(\ds \prod_{\substack{i=0 \\ i\ne j}}^{\ell} (\rho_j -\rho_i)\right)^{-1},
\label{eq:bary_poly}
\end{align}
If $f$ is even or odd and the interpolation points are symmetric about the origin, then we can exploit these symmetries to reduce the sums in determining $v$ and the products in computing the weights to be over roughly half the terms.  In deriving these formulas, we assume, without loss of generality, that $1 \geq \rho_0 > \rho_1 \cdots > \rho_{\ell} \geq -1$, and that $\ell=2n$ or $\ell=2n+1$.  In the former case, the number of interpolation points is odd, and the assumed symmetry about the origin implies that the origin is included in the point set and $\rho_n=0$.  In the latter case, the origin is not included.  In either case, symmetry about the origin implies $\rho_{j}=-\rho_{\ell-j}$, $j=0,\ldots,n$.  Similar to Section \ref{sec:even_oddt}, we let $\mathcal{S}_n = \{\rho_j\}_{j=0}^n$ be the set of interpolation points that are actually used in the formulas.

We first consider the case that $f$ is even.  To derive the barycentric formula for this case, it is important to distinguish whether $\ell$ is even or odd.  Assuming first that $\ell=2n$, so that the origin is included,  it is straightforward to show that for $j=0,\ldots,n-1$ the barycentric weights satisfy $w_{j} = w_{\ell-j}$, and the function samples satisfy $f_{j} = f_{\ell-j}$.  Applying these results to \eqref{eq:bary_poly} and denoting the interpolant as $\ue$ gives:
\begin{align}
\ue(\rho) &=  \dfrac{\ds \frac{w_n}{\rho - \rho_n}f_n + \sum_{j=0}^{n-1} \left[\frac{w_j}{\rho - \rho_j}f_j + \frac{w_{\ell-j}}{\rho -\rho_{\ell-j}}f_{\ell-j}\right]}{\ds \frac{w_n}{\rho - \rho_n} + \sum_{j=0}^{n-1} \left[\frac{w_j}{\rho - \rho_j}+\frac{w_{\ell-j}}{\rho -\rho_{\ell-j}}\right]}
	=
\dfrac{\ds \frac{w_n}{\rho}f_n + \sum_{j=0}^{n-1} w_j \left[\frac{1}{\rho - \rho_j}+\frac{1}{\rho + \rho_j}\right]f_j}{\ds   \frac{w_n}{\rho} + \sum_{j=0}^{n-1} w_j \left[\frac{1}{\rho - \rho_j}+\frac{1}{\rho + \rho_j}\right]} \nonumber \\
        & = 
\dfrac{\ds\frac{\rho w_n}{\rho^2}f_n+\sum_{j=0}^{n-1} \dfrac{2\rho w_j}{\rho^2-\rho_j^2}f_j}{\ds\frac{w_n\rho}{\rho^2}+\sum_{j=0}^{n-1} \dfrac{2\rho w_j}{\rho^2-\rho_j^2}} = \dfrac{\ds2\rho\left(\frac{w_n}{2\rho^2}f_n+\sum_{j=0}^{n-1} \dfrac{w_j}{\rho^2-\rho_j^2}f_j\right)}{\ds2\rho\left(\frac{w_n}{2\rho^2}+\sum_{j=0}^{n-1} \dfrac{w_j}{\rho^2-\rho_j^2}\right)} = 
\dfrac{\ds\sideset{}{'}\sum_{j=0}^{n} \dfrac{w_j}{\rho^2-\rho_j^2}f_j}{\ds\sideset{}{'}\sum_{j=0}^{n} \dfrac{w_j}{\rho^2-\rho_j^2}},
\label{eq:bary_even_leven}
\end{align}
where the prime on the summations means that the last term is halved.  One can also simplify the barycentric weights in \eqref{eq:bary_poly} to products involving only $\rho_0, \ldots, \rho_n$:
%\begin{align*}
%	w_j = \begin{cases}
%		\left(\ds 2\rho_j^2\prod_{i=0, i\ne j}^{n-1} \left(\rho_j^2 -\rho_i^2\right)\right)^{-1}, & j\neq n, \\
%		(-1)^{n-1}\left(\ds \prod_{i=0}^{n-1}\rho_i^2\right)^{-1}, & j = n. 
%	\end{cases}
%\end{align*}
\begin{align}
	w_j = \begin{cases}
		\left(\ds 2\prod_{\substack{i=0 \\ i\ne j}}^{n} \left(\rho_j^2 -\rho_i^2\right)\right)^{-1}, & j\neq n, \\
		(-1)^{n-1}\left(\ds \prod_{i=0}^{n-1}\rho_i^2\right)^{-1}, & j = n. 
	\end{cases}
	\label{eq:barywghts_even_leven}
\end{align}
These formulas can be further simplified by eliminating the common 2 in all the weights, using $\rho_n = 0$, and incorporating the division by 2 into the weights when $j=n$ in the summations to obtain:
%\begin{align}
%u(\rho) = \dfrac{\ds\sum_{j=0}^{n} \dfrac{w_j}{\rho^2-\rho_j^2}f_j}{\ds\sum_{j=0}^{n} \dfrac{w_j}{\rho^2-\rho_j^2}},\quad 
%	w_j = \begin{cases}
%		\left(\ds \prod_{i=0, i\ne j}^{n} \left(\rho_j^2 -\rho_i^2\right)\right)^{-1}, & j\neq n, \\
%		(-1)^{n-1}\left(\ds \prod_{i=0}^{n-1}\rho_i^2\right)^{-1}, & j = n,
%	\end{cases}
%\label{eq:bary_even_leven}
%\end{align}
\begin{align}
\ue(\rho) = \dfrac{\ds\sum_{j=0}^{n} \dfrac{w_j}{\rho^2-\rho_j^2}f_j}{\ds\sum_{j=0}^{n} \dfrac{w_j}{\rho^2-\rho_j^2}},\quad 
	w_j = \left(\ds \prod_{\substack{i=0 \\ i\ne j}}^{n} \left(\rho_j^2 -\rho_i^2\right)\right)^{-1},
\label{eq:bary_even}
\end{align}
Note that while we used the ordering $\rho_j > \rho_{j+1}$, $j=0,\ldots,n-1$ to derive \eqref{eq:bary_even}, this ordering is not necessary to use the formula.  

If $\ell=2n+1$, then the origin is not included in the point set and it is straightforward to show that the barycentric weights in \eqref{eq:bary_poly} satisfy $w_{j} = -w_{\ell-j}$, $j=0,\ldots,n$.  Using this symmetry, and following a similar approach to the derivations of \eqref{eq:bary_even_leven} and \eqref{eq:barywghts_even_leven} produces the following formulas for the interpolant and weights:
\begin{align*}
\ue(\rho) = \dfrac{\ds\sum_{j=0}^n \dfrac{w_j\rho_j}{\rho^2-\rho_j^2}f_j}{\ds\sum_{j=0}^n \dfrac{w_j\rho_j}{\rho^2-\rho_j^2}},\quad 
w_j = \left(\rho_j\prod_{\substack{i=0 \\ i\ne j}}^{n} \left(\rho_j^2 - \rho_i^2\right)\right)^{-1}.
%\label{eq:bary_even_lodd}
\end{align*}
%Note that one could divide out the $\rho_j$ factor in the denominator of the formula for the weights with the same factor in the numerator of the summations defining $u$.  However, we leave these factors here since then the standard weights for Chebyshev points of types I and II (see below) do not need to be modified when using this formula for $u$.
By dividing out the common  $\rho_j$ factor in the denominator of the formula for the weights with the same factor in the numerator of the summations defining $\ue$, we obtain the same barycentric formula \eqref{eq:bary_even} for the interpolant and weights. Thus, \eqref{eq:bary_even} can be used in both cases of $\ell = 2n$ or $\ell=2n+1$, which correspond to $0\in\mathcal{S}_n$ and $0\notin\mathcal{S}_n$, respectively.

The barycentric formula for the case where $f$ is an odd function, denoted by $\uo(\rho)$, can be derived using the same procedure as the even case, but exploiting the odd symmetry in the function samples: $f_j = -f_{\ell-j}$, $j=0,\ldots,n$.  However, we derive it using a simpler approach that relies on the even formula \eqref{eq:bary_even}; this is the analogous the approach used by Berrut~\cite{berrut1984baryzentrische} to derive the sine interpolation formula \eqref{eq:oddt} from the cosine formula~\eqref{eq:event}.  If $\ell = 2n$, so that $0\in\mathcal{S}_n$, then the unique odd degree polynomial that interpolates the data must be of degree $2n-1$.  This interpolant can be constructed by first forming the degree $2n$ even polynomial interpolant \eqref{eq:bary_even} to the even data $f_j\rho_j$, $j=0,\ldots,n,$ and then dividing it by $\rho$, which gives
\begin{align}
\uo(\rho) = 
\dfrac{1}{\rho}
\dfrac{\ds\sum_{j=0}^{n} \dfrac{w_j\rho_j}{\rho^2-\rho_j^2}f_j}{\ds\sum_{j=0}^{n} \dfrac{w_j}{\rho^2-\rho_j^2}},\;\text{if}\; 0\in\mathcal{S}_n.
\label{eq:bary_odd_origin1}
\end{align}
If instead $\ell=2n+1$, so that $0\notin\mathcal{S}_n$, then the unique odd interpolant to $f_j$ will be of degree $2n+1$.  So, we again use the unique degree $2n$ even polynomial interpolant \eqref{eq:bary_even}, but now for the data $f_j/\rho_j$, $j=0,\ldots,n,$ and then multiply it by $\rho$:
\begin{align}
\uo(\rho) = 
\rho\dfrac{\ds \sum_{j=0}^{n} \dfrac{w_j}{\rho_j(\rho^2-\rho_j^2)}f_j}{\ds\sum_{j=0}^{n} \dfrac{w_j}{\rho^2-\rho_j^2}},\;\text{if}\; 0\notin\mathcal{S}_n, 
\label{eq:bary_odd_noorigin}
\end{align}
%\begin{align}
%v(\rho) = 
%\begin{cases}
%\dfrac{1}{\rho}
%\dfrac{\ds\sum_{j=0}^{n} \dfrac{w_j\rho_j}{\rho^2-\rho_j^2}f_j}{\ds\sum_{j=0}^{n} \dfrac{w_j}{\rho^2-\rho_j^2}}, & 0\in\mathcal{S}_n, \\
%\rho\dfrac{\ds \sum_{j=0}^{n} \dfrac{w_j}{\rho_j(\rho^2-\rho_j^2)}f_j}{\ds\sum_{j=0}^{n} \dfrac{w_j}{\rho^2-\rho_j^2}}, & 0\notin\mathcal{S}_n, 
%\end{cases}
%\label{eq:bary_odd}
%\end{align}

An equivalent interpolation formula to \eqref{eq:bary_odd_origin1} can be obtained without a division by $\rho$ by using the result that \eqref{eq:bary_odd_origin1} is exact for the function $f(\rho) = \rho$.  Combining this formula with \eqref{eq:bary_odd_noorigin}, we obtain the following barycentric formula for samples with odd symmetry:
\begin{align}
\uo(\rho) = 
\begin{cases}
\rho
\dfrac{\ds\sum_{j=0}^{n} \dfrac{w_j\rho_j}{\rho^2-\rho_j^2}f_j}{\ds\sum_{j=0}^{n} \dfrac{w_j \rho_j^2}{\rho^2-\rho_j^2}}, & 0\in\mathcal{S}_n, \\
\rho\dfrac{\ds \sum_{j=0}^{n} \dfrac{w_j}{\rho_j(\rho^2-\rho_j^2)}f_j}{\ds\sum_{j=0}^{n} \dfrac{w_j}{\rho^2-\rho_j^2}}, & 0\notin\mathcal{S}_n, 
\end{cases}
\label{eq:bary_odd}
\end{align}
where the barycentric weights are given in \eqref{eq:bary_even}.

When the points $\mathcal{S}_n$ correspond to Chebyshev points of the first or second kind, the barycentric weights $w_j$ given over $[-1,1]$ in \eqref{eq:bary_poly} can be computed explicitly~\cite{berrut2004barycentric} and the interpolation formulas \eqref{eq:bary_even} and \eqref{eq:bary_odd} can then be simplified.   We have summarized these formulas in the first two rows of Table \ref{tbl:radial_formulas} for the specific case of interpolating in the radial direction on the disk.  

As noted at the end of Section \ref{sec:even_oddt}, the barycentric weights in \eqref{eq:bary_poly} for the GL points can be computed efficiently (and stably) using the \texttt{legpts} function of Chebfun~\cite{driscoll2014chebfun}.  These weights can then be appropriately modified for incorporation into \eqref{eq:bary_even} and \eqref{eq:bary_odd}.

\subsection{Barycentric formulas for the disk\label{sec:bary_formula_disk}}
Let $f_{j,k}$ denote samples of a function $f(\phi,\rho)$ on the tensor product grid $(\phi_k,\rho_j)$, $j=0,\ldots,n$, $k=0,\ldots,2m-1$. Then the samples of $f^-(\phi,\rho)$ and $f^+(\phi,\rho)$ from Lemma \ref{lemma:bmcII} are given as 
\begin{align}
f_{j,k}^{-} = \frac12(f_{j,k}-f_{j,k+m}) \;\text{and}\; f_{j,k}^{+} = \frac12(f_{j,k}+f_{j,k+m}),\; j=0,\ldots,n,\; k=0,\ldots,m-1.
\label{eq:fplus_fminus_disk}
\end{align}  
We use the same approach as the sphere to interpolate these sets of data and construct bivariate interpolants from tensor products of the 1D anti-periodic/periodic trigonometric formulas from Section \ref{sec:periodic_aperiodic} with the 1D odd/even polynomial formulas from the previous section.  The details in the derivations of the disk formulas are similar to the sphere formulas, so we omit them and just state the result:
%We also apply the analogue of the double Fourier sphere method for the disk to derive a barycentric interpolation formula for the disk. Since the underlining idea is the same as that of the sphere, we derive the interpolation formulas for the disk in the same way. First we derive formulas for $f^-$ and $f^+$ on $(\phi,\rho) \in [0,\pi]\times [0,1]$, then combine them accordingly. In the angular direction we use the same formula as the sphere, but we use rational polynomial barycentric interpolation formulas for the radial direction. Therefore, the barycentric interpolation formula for the disk is given as 
\begin{align}
\label{eq:bary_disk}
\disk(\phi , \rho) = 
\begin{cases}
\dfrac{\displaystyle{\sum_{k=0}^{m-1} (-1)^k \left[\cot \left(\phi - \phi_k\right)\ue_k(\rho) + \csc \left(\phi - \phi_k\right)\uo_k(\rho)\right]}}{\displaystyle{\sum_{k=0}^{m-1} (-1)^k\cot \left(\phi - \phi_k\right)}}, & m \text{ even }\\
\dfrac{\displaystyle{\sum_{k=0}^{m-1}(-1)^k \left[\csc\left(\phi - \phi_k\right)\ue_k(\rho) + \cot \left(\phi - \phi_k\right)\uo_k(\rho)\right]}}
{\displaystyle{\sum_{k=0}^{m-1}(-1)^k \csc \left(\phi - \phi_k\right)}}, & m \text{ odd },
\end{cases}
\end{align}
where the radial interpolants $\ue_k(\rho)$ and $\uo_k(\rho)$ are given in the Table \ref{tbl:radial_formulas}. 
\begin{table}[htb]
\centering
\begin{tabular}{|c||c|c||c|c|}
\hline
Grid & \multicolumn{2}{c||}{Origin included} & \multicolumn{2}{c|}{Origin not included}\\
\hline
 & $\ue_k(\rho)$ & $\uo_k(\rho)$ & $\ue_k(\rho)$ & $\uo_k(\rho)$ \\
\hline
\hline
\hline
CH1 & $\ds \dfrac{\ds \sideset{}{^{\prime}}\sum_{j=0}^n \dfrac{(-1)^j \xi_j}{\rho^2 - \rho_j^2}f_{j,k}^{+}}{\ds\sideset{}{^{\prime}}\sum_{j=0}^n \dfrac{(-1)^j \xi_j}{\rho^2 - \rho_j^2}}$ & 
$\rho\dfrac{\ds\sideset{}{^{\prime}}\sum_{j=0}^{n} \dfrac{(-1)^j\xi_j\rho_j}{\rho^2-\rho_j^2}f_{j,k}^{-}}{\ds\sideset{}{^{\prime}}\sum_{j=0}^{n} \dfrac{(-1)^j\xi_j\rho_j^2}{\rho^2-\rho_j^2}}$ & 
$\ds \dfrac{\ds \sideset{}{}\sum_{j=0}^n \dfrac{(-1)^j \rho_j\gamma_j}{\rho^2 - \rho_j^2}f_{j,k}^{+}}{\ds\sideset{}{}\sum_{j=0}^n \dfrac{(-1)^j \rho_j\gamma_j}{\rho^2 - \rho_j^2}}$ & 
$\rho\dfrac{\ds\sideset{}{}\sum_{j=0}^{n} \dfrac{(-1)^j\gamma_j}{\rho^2-\rho_j^2}f_{j,k}^{-}}{\ds\sideset{}{}\sum_{j=0}^{n} \dfrac{(-1)^j\gamma_j\rho_j}{\rho^2-\rho_j^2}}$ \\
\cline{2-5}
& 
\multicolumn{2}{c||}{where $\ds\xi_j = \sin\left(\dfrac{2j+1}{4n+2}\pi\right)$} &\multicolumn{2}{c|}{where $\gamma_j = \sin\left(\dfrac{2j+1}{4n+4}\pi\right)$ } \\
\hline
\hline
CH2 & $\ds \dfrac{\ds \sideset{}{_{\prime}'}\sum_{j=0}^n \dfrac{(-1)^j}{\rho^2 - \rho_j^2}f_{j,k}^{+}}{\ds \sideset{}{_{\prime}'}\sum_{j=0}^n \dfrac{(-1)^j}{\rho^2 - \rho_j^2}}$ & 
$\ds \rho\dfrac{\ds\sideset{}{_{\prime}'}\sum_{j=0}^{n} \dfrac{(-1)^j\rho_j}{\rho^2-\rho_j^2}f_{j,k}^{-}}{\ds\sideset{}{_{\prime}'}\sum_{j=0}^{n} \dfrac{(-1)^j\rho_j^2}{\rho^2-\rho_j^2}}$ & 
$\ds \dfrac{\ds \sideset{}{_{\prime}}\sum_{j=0}^n \dfrac{(-1)^j\rho_j}{\rho^2 - \rho_j^2}f_{j,k}^{+}}{\ds \sideset{}{_{\prime}}\sum_{j=0}^n \dfrac{(-1)^j\rho_j}{\rho^2 - \rho_j^2}}$ & 
$\ds \rho\dfrac{\ds\sideset{}{_{\prime}}\sum_{j=0}^{n} \dfrac{(-1)^j}{\rho^2-\rho_j^2}f_{j,k}^{-}}{\ds\sideset{}{_{\prime}}\sum_{j=0}^{n} \dfrac{(-1)^j\rho_j}{\rho^2-\rho_j^2}}$ \\
\hline 
\hline
GL & $\ds \dfrac{\ds\sum_{j=0}^{n} \dfrac{w_j}{\rho^2-\rho_j^2}f_j}{\ds\sum_{j=0}^{n} \dfrac{w_j}{\rho^2-\rho_j^2}}$ & 
$\ds \rho \dfrac{\ds\sum_{j=0}^{n} \dfrac{w_j\rho_j}{\rho^2-\rho_j^2}f_j}{\ds\sum_{j=0}^{n} \dfrac{w_j \rho_j^2}{\rho^2-\rho_j^2}}$ & 
$\ds \dfrac{\ds\sum_{j=0}^{n} \dfrac{w_j}{\rho^2-\rho_j^2}f_j}{\ds\sum_{j=0}^{n} \dfrac{w_j}{\rho^2-\rho_j^2}}$ &  
$\ds \rho\dfrac{\ds \sum_{j=0}^{n} \dfrac{w_j}{\rho_j(\rho^2-\rho_j^2)}f_j}{\ds\sum_{j=0}^{n} \dfrac{w_j}{\rho^2-\rho_j^2}}$ \\
\hline
\end{tabular}
\caption{Barycentric polynomial interpolants in the radial direction of the disk to the data $f_{j,k}^{+}$ and $f_{j,k}^{-}$ given in \eqref{eq:fplus_fminus_disk} for the various grids under consideration. The lower prime on a summation sign means the first term is halved, while an upper prime means the last term is halved.  The formulas for $w_{j}$ are given in \eqref{eq:bary_even}.  In the CH1 and CH2 cases, the points are assumed to be arranged so that $\rho_j > \rho_{j+1}$, $j=0,\ldots,n-1$.\label{tbl:radial_formulas}}
\end{table}
\begin{theorem}
The interpolant $\disk:[0,2\pi]\times[0,1]\rightarrow\mathbb{C}$ to the data $f_{jk}$, $j=0,\ldots,n$, $k=0,\ldots,2m-1$, sampled from a continuous function $f$ on the disk at the CH1, CH2, or GL grids, is a BMC function.  Furthermore, in the case where the grid contains $\rho=0$, $\disk$ is a BMC-II function.
\end{theorem}
\begin{proof}
The first result follows naturally from Lemma \ref{lemma:bmcII} and the second follows using the same arguments as the second part of the proof of Theorem \ref{thm:bmc_sphere}.
\end{proof}

\begin{remark}
Similar to the sphere, the interpolant \eqref{eq:bary_disk} is also more generally applicable to other tensor product grids that use equally spaced polar points, but different radial points than CH1, CH2, and GL.  The formulas in this more general case are based on $\ue_k$ and $\uo_k$ given in the third row of Table \ref{tbl:radial_formulas}.  If these radial points contain $0$, then the interpolant will be a BMC-II function.
\end{remark}

%\begin{remark}
%Floater-Hormann
%\end{remark}

Note that if the grids do not contain the origin as interpolation points, then the interpolant \eqref{eq:bary_disk} is not guaranteed to be single-valued there.  As with the sphere grids not containing the poles, this could lead to issues when trying to differentiate the interpolated values near $\rho=0$.  However, we have not observed that the convergence rates of the interpolants are degraded in this case.  
%A single-valued result can be obtained by using the average of all the values of $u_k(0)$ when evaluating $D(\phi,0)$.

\subsection{Numerical example}
We use a similar numerical experiment as the sphere to demonstrate the barycentric interpolation formulas \eqref{eq:bary_disk} work for the disk.  For this case, we interpolate the target function
\begin{equation}
	\label{eq:fund}
	f(\phi,\rho)=\sin(21\pi(1+\cos(\pi\rho))(\rho^2-2\rho^5\cos(5(\phi-0.11)))),
\end{equation}
which is displayed in Figure \ref{fig:errs_disk}.  We again select $n = m$, giving $N=2n(n+1)$ total grid points and samples of $f$.  We compute the relative max-norm error between $\disk$ and $f$ from a dense set of scattered evaluation points over the disk.  These errors are displayed as a function of $m$ for each of the three grids under consideration in Figure \ref{fig:errs_disk} (b).  We see the max-norm errors are nearly the same for each of the three different grid types and that the errors decrease exponentially fast with $m$, which is expected since $f$ is infinitely smooth.

\begin{figure}[ht]
	\centering
	\begin{tabular}{cc}
	\includegraphics[width=0.4\linewidth]{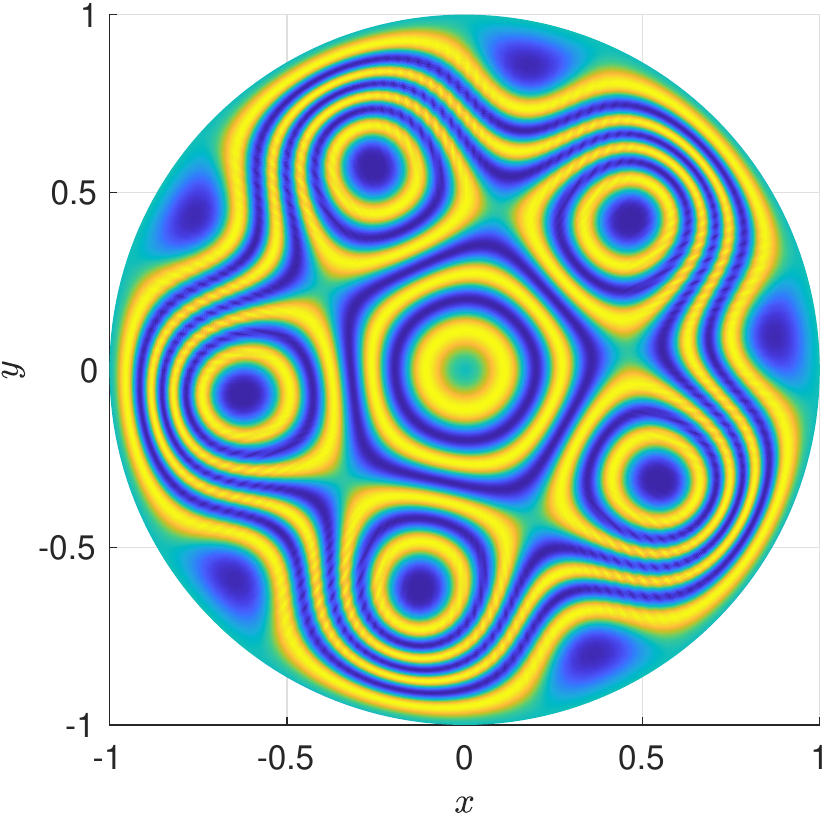} & \includegraphics[width=0.51\linewidth]{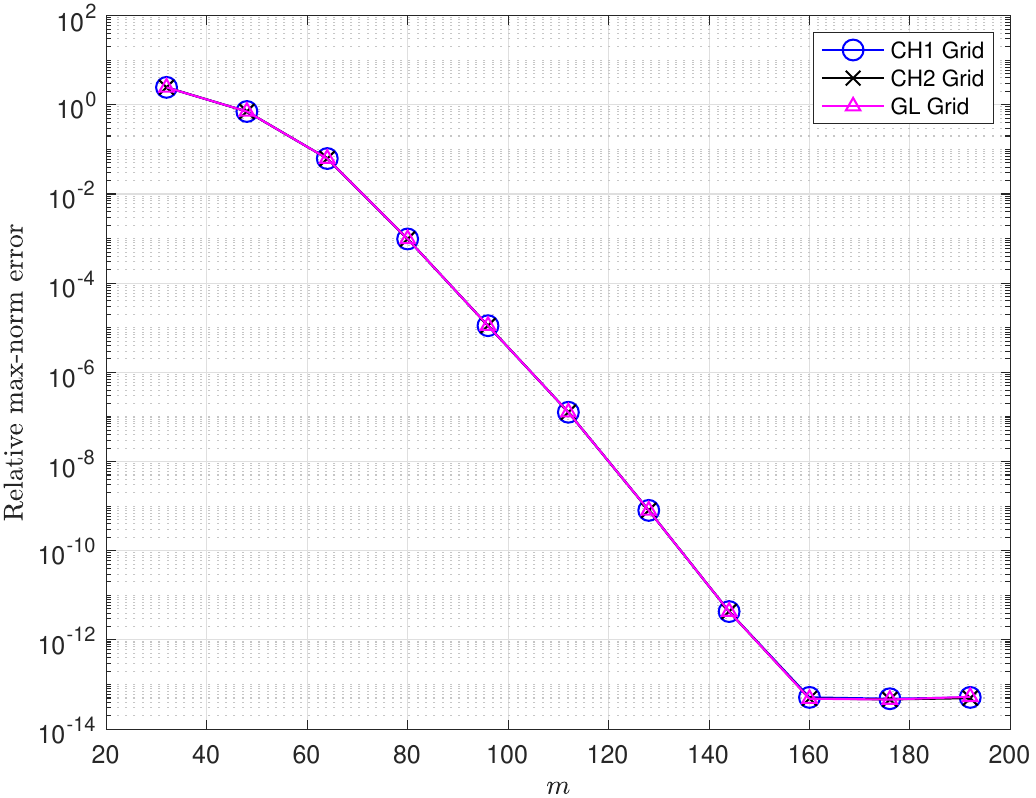} \\
	(a) & (b) 
	\end{tabular}
	\caption{(a) Test function \eqref{eq:fund} on the disk, where dark blue is corresponds to -1 and bright yellow to 1. (b) Relative max-norm error in the barycentric interpolants \eqref{eq:bary_disk} of the function in (a) for the CH1, CH2, and GL grids using $n = m$.\label{fig:errs_disk}}
\end{figure}

\section{Semi-Lagrangian advection (SLA) on the sphere\label{sec:application}}
In this section we combine the new spherical barycentric interpolation formulas with a SLA method to solve the tracer transport equation on the sphere. For the unit sphere $\Sph$, this equation is given 
\begin{align}
\frac{\partial q}{\partial t} + \vu \cdot \nabla_{\Sph} q = 0,\quad \text{or}\quad \frac{Dq}{Dt} &= 0, \label{eq:advection}
\end{align}
where $q$ is the scalar quantity being transported, $\vu$ is an incompressible velocity field that is tangent to $\Sph$, and $\nabla_{\Sph}$ denotes the surface gradient operator on $\Sph$.  SLA techniques have been widely used in atmospheric sciences for simulating tracer transport and more complex non-linear shallow water flows~\cite{staniforth1991SL}. The appeal of SLA lies in its ability to avoid the strict time-step limitations of Eulerian methods, allowing for significantly larger time-steps than those dictated by the CFL condition, all while maintaining stability~\cite{FalconeFerrettiBook}.  We follow a similar approach to~\cite{shankar_2018}  for testing new interpolation methods with SLA for solving \eqref{eq:advection}.

The SLA scheme we consider operates as follows: we assume that a Lagrangian particle reaches each point on a fixed Eulerian grid $X = \{(\phi_j,\theta_j)\}_{j=1}^N$ at some time $t+\Delta t$, carrying a certain amount of the scalar $q$. The amount of $q$ at $(\phi_j,\theta_j)$ must have been transported in the flow field $\vu$ from the particle’s \textit{departure point} $(\phi_j^{\rm d},\theta_j^{\rm d})$ at time $t$. To find the {departure point}, we trace the particle backward in time from $t+\Delta t$ to $t$ along the flow field $\vu$, then determine $q$ by interpolating the scalar field to that departure point.  Please see Algorithm \ref{alg:sl_alg} for an outline of this scheme. Our primary focus is on the interpolation step, while the particle trajectory will be computed using a fifth-order Runge-Kutta method (see~\cite{shankar_2018} for details on this step).  We use this SLA scheme with each of the three grids from Section \ref{sec:tensor_product} and use the corresponding barycentric formula  \eqref{eq:bary_sphere} for the interpolation step.

\begin{algorithm}
\caption{SLA for tracer transport on $\Sph$ using barycentric interpolation}
\label{alg:sl_alg}
\begin{algorithmic}
        \State \textbf{Input:} 1) $\vu(\phi,\theta,t)$, incompresible velocity field tangent to $\Sph$; 2)  $q(\phi,\theta,0)$, initial tracer field; 3) $X = \{(\phi_j,\theta_j)\}_{j=1}^N$, Eulerian tensor product grid;
	4) $t_{\rm f}$, final time; 4) $\Delta t$, time-step.
	\State Initailize $\vq_{X}^{0} = \{q(\phi_j,\theta_j,0)\}_{j=1}^N$, $t=0$, and $k=0$.
	\While{$t \leq t_{\rm f}$}
		\State \textit{Particle-trajectory}: Trace back Lagrangian particle at grid point $\vx_j=(\phi_j,\theta_j)$ in the 
		\State  velocity field from $t+\Delta t$ to $t$ to find its departure point $\vx_j^{\rm d} = (\phi_j^{\rm d},\theta_j^{\rm d})$, for $j=1,\ldots,N$.
		\State \textit{Interpolation}: Interpolate $\vq^{k}_{X}$ to $X^{\rm d}=\{(\phi_j^{\rm d},\theta_j^{\rm d})\}_{j=1}^N$ using barycentric interpolation 
		\State formula \eqref{eq:bary_sphere} to obtain $\vq_{X^{\rm d}}^k$.
		\State Update $\vq^{k+1}_X =  \vq_{X^{\rm d}}^k$, $k=k+1$, and $t =  k\Delta t$.
	\EndWhile
\end{algorithmic}
\end{algorithm}

%\begin{algorithm}
%\caption{SLA for tracer transport on $\Sph$ using barycentric interpolation}
%\label{alg:sl_alg}
%\begin{algorithmic}
%        \State \textbf{Input:} 1) $\vu(\vx,t)$: incompresible velocity field tangent to $\Sph$; 2)  $q(\vx,0)$: initial tracer field; 3) $X = \{\vx_j\}_{j=1}^N$: Eulerian latitude-longitude grid;
%	4) $t_{\rm f}$: final time; 4) $\Delta t$: time-step.
%	\State Set $\vq_{X}^{0} = \{q(\vx_j,0)\}_{j=1}^N$, $t=0$, and $m=0$.
%	\While{$t \leq t_{\rm f}$}
%		\State \textit{Particle-trajectory}: For $j=1,\ldots,N$, trace back $\vx_j$ to time $t$ to find departure point $\vxi^m_j$.
%		\State \textit{Interpolation}: Interpolate $\vq^{m}_{X}$ to $\Xi^{m}=\{\vxi_j^{m}\}_{j=1}^N$ using barycentric interpolation to obtain $\vq_{\Xi}^m$.
%		\State Set $\vq^{m+1}_X =  \vq_{\Xi}^m$, $m=m+1$, and $t =  m\Delta t$.
%	\EndWhile
%\end{algorithmic}
%\end{algorithm}

\subsection{Results}
We consider the popular deformational flow test case from~\cite{NairLauritzen2010} to test the barycentric SLA scheme. The velocity field for this test is given as
\begin{align}
u(\lambda,\theta,t) &= \frac{10}{T}\cos\lf(\frac{\pi t}{T}\rt)\sin^2\lf(\lambda - \frac{2\pi t}{T}\rt)\sin (2\theta) + \frac{2\pi}{T}\cos(\theta), \\
v(\lambda,\theta,t) &= \frac{10}{T}\cos\lf(\frac{\pi t}{T}\rt)\sin\lf(2\lambda - \frac{2\pi t}{T}\rt)\cos (\theta),
\end{align}
where $u$/$v$ are components of the velocity field in longitude/latitude and $T=5$. This velocity field translates and deforms the initial condition $q(\lambda,\theta,0)$ up to time $t=2.5$ and then reverses so that $q$ is returned to its initial position and value at time $t=5$, i.e., $q(\lambda,\theta,0)=q(\lambda,\theta,5)$. The initial condition then serves as the exact solution for comparing the numerical solutions at $t=5$. 
% In the experiments, we convert the velocity field to the Cartesian basis and use that for computing the particle trajectories (see~\cite{shankar_2018} for more details). 
\begin{figure}[t]
\centering
\begin{tabular}{c|c}
Cosine bells & Gaussian bells\\
\hline
\includegraphics[width=0.5\textwidth]{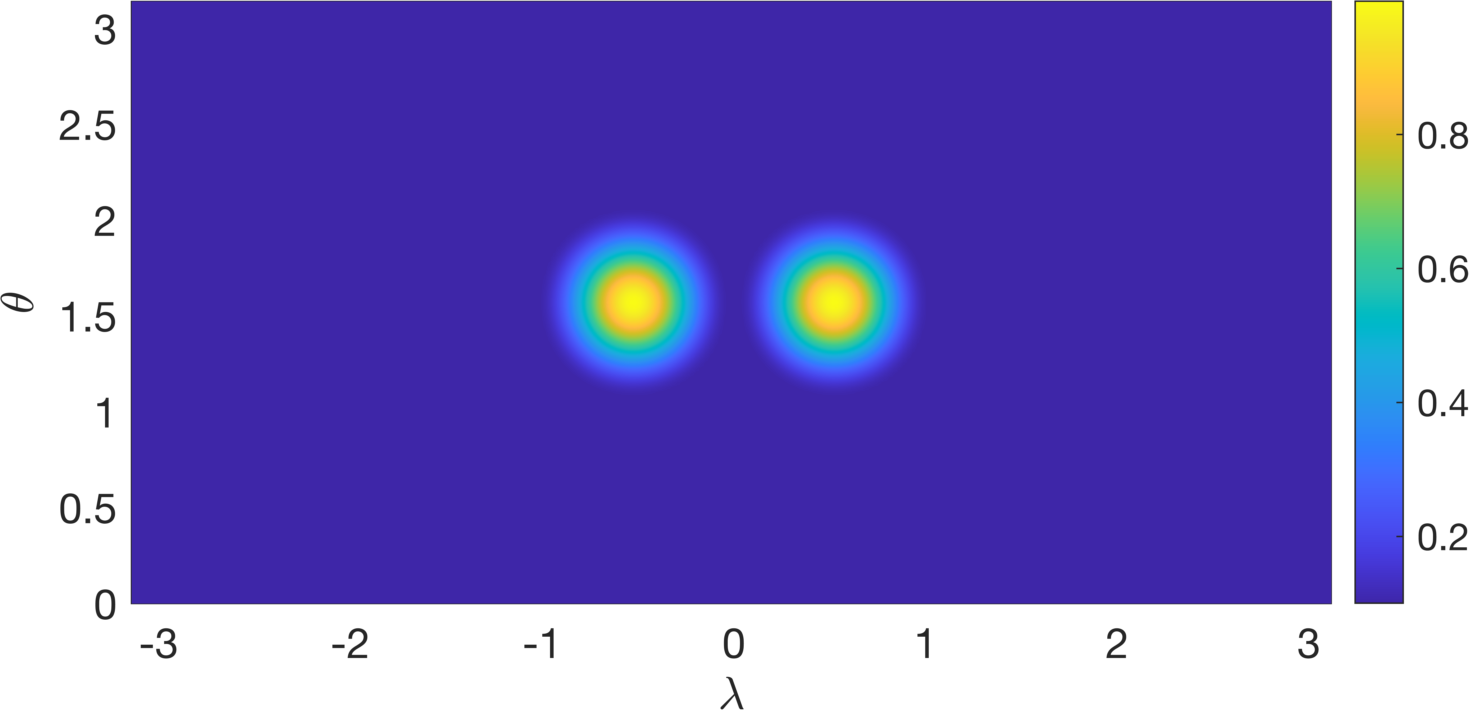} 
& \includegraphics[width=0.5\textwidth]{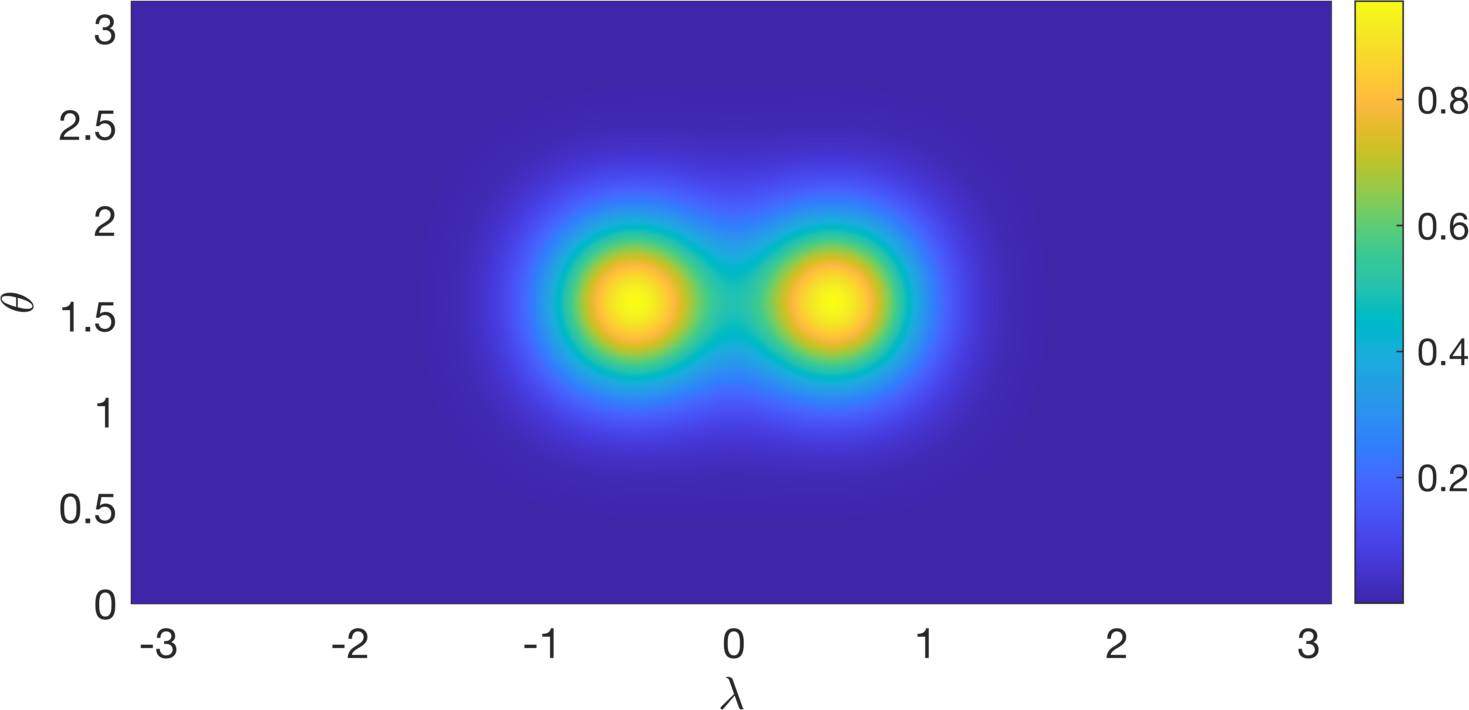} \\
\multicolumn{2}{c}{Solution at $t=0$ and $t=5$} \\
\includegraphics[width=0.5\textwidth]{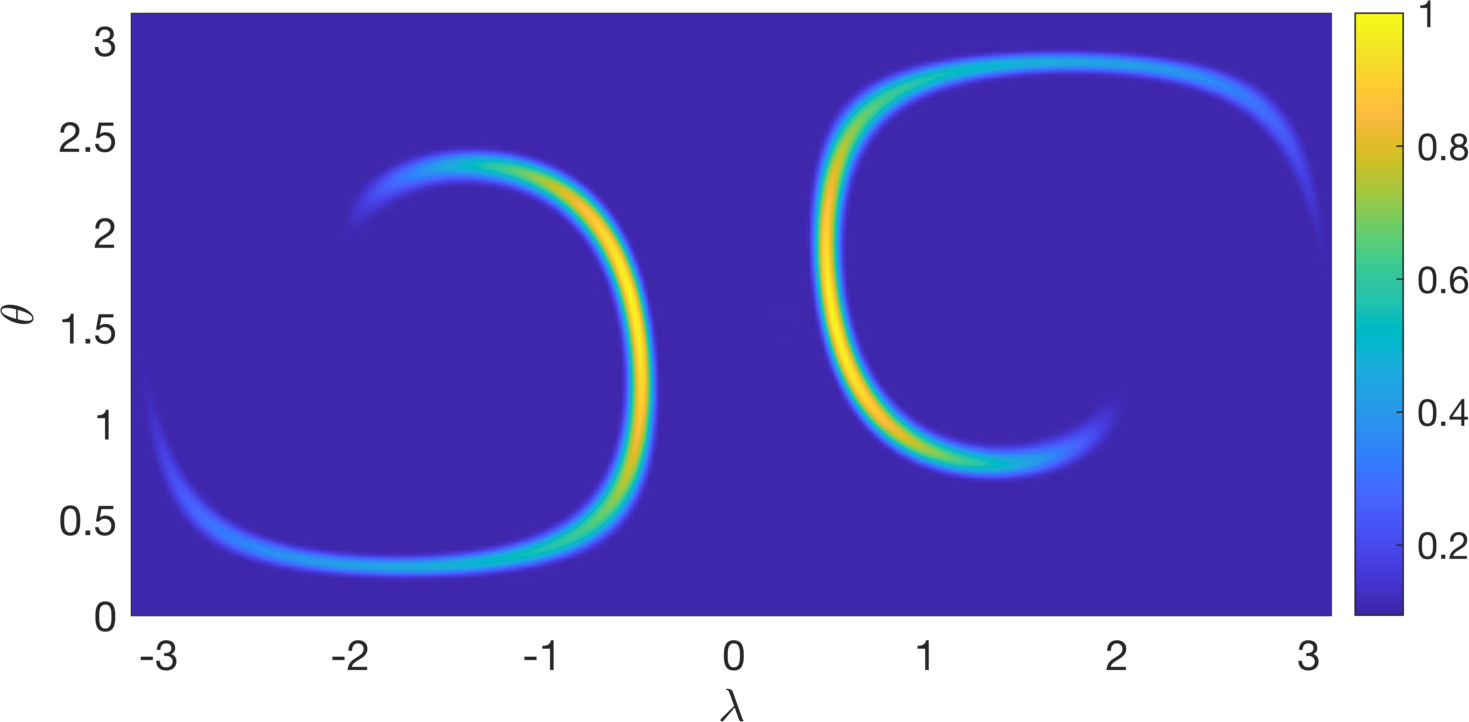} 
& \includegraphics[width=0.5\textwidth]{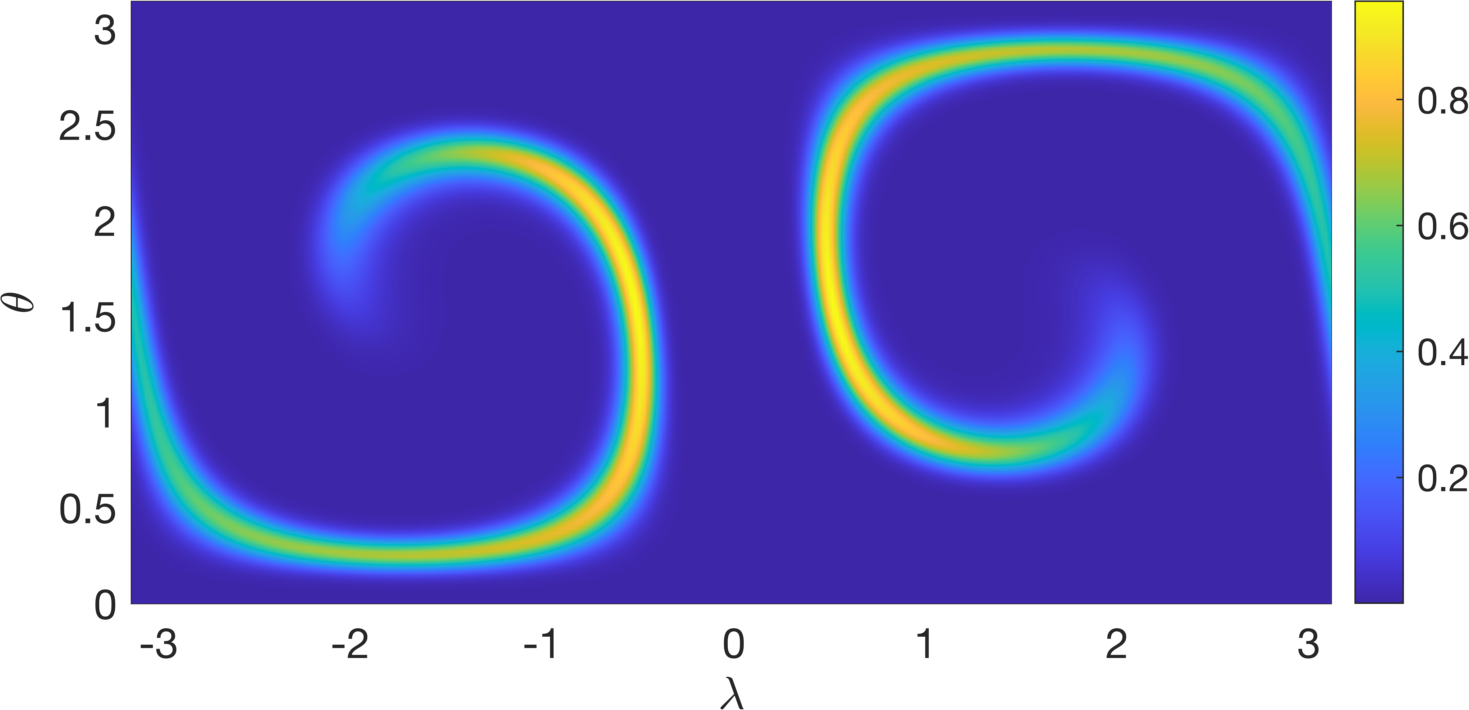} \\
\multicolumn{2}{c}{Solution at $t=2.5$}
\end{tabular}
\caption{Illustrations of the solutions at $t=0$, $t=2.5$, and $t=5$ (same as $t=0$) for the deformational flow test case. \label{fig:ics}}
\end{figure}
We consider the following two initial conditions for the experiments:
\begin{itemize}
\item \textit{Cosine bells}: $q(\lambda,\theta,0) = 0.1 + 0.9(q_1(\lambda,\theta) + q_2(\lambda,\theta))$,
where for $i=1,2$
\begin{align*}
q_i(\lambda,\theta) &=
\begin{cases}
\frac{1}{2}\lf(1 + \cos\lf(2\pi \arccos\lf(r_i(\lambda,\theta)\rt)\rt)\rt), & \arccos\lf(r_i(\lambda,\theta)\rt) < \frac{1}{2}, \\
0 & \text{otherwise}.
\end{cases}
\end{align*}
\item \textit{Gaussian bells}:  $
q(\lambda,\theta,0) = 0.95\lf( e^{-10(1-r_1(\lambda,\theta))} + e^{-10(1-r_2(\lambda,\theta))} \rt)$.
\end{itemize}
Here $r_i(\lambda,\theta) = \cos(\theta)\cos(\theta_i) + \cos(\lambda-\lambda_i)\sin(\theta)\sin(\theta_i)$, $i=1,2$, and the centers of the bells are given as $(\lambda_1,\theta_1)  = (\pi/6,0)$ and $(\lambda_2,\theta_2)  = (-\pi/6,0)$.  The first initial condition is only $C^1(\Sph)$ and is meant to test the dispersive errors of a numerical solver, while the second is $C^{\infty}(\Sph)$ and is meant to show the maximum convergence rate that is possible.  Plots of the solutions of the deformational flow test case at $t=0$, $2.5$, and $5$ are displayed in Figure \ref{fig:ics}.

%In the tests, we use an equally spaced latitude-longitude grid with spacing $h=2\pi/n$ in both directions. For the cosine bells experiments, we use a time-step of $\Delta t = 1/10$, while for the Gaussian bells, we use $\Delta t = 1/20$ (except for the $n=1024$ result where we use $\Delta t=1/40$).  We compare the DFS/NUFFT interpolation with a bicubic interpolation on a latitude-longitude grid since this is the most popular interpolation scheme used in SLA methods applied to atmospheric-type problems~\cite{staniforth1991SL}.  Plots of the errors in both methods at $t=5$ and different $n$ for the cosine bells experiment are shown in Figure \ref{fig:cosine}, while the relative errors in the solution and the mass of the solution are shown in Table \ref{tbl:cosine}.  Similar results for the Gaussian initial conditions are shown in Figure \ref{fig:gaussian} and Table \ref{tbl:gaussian}.

We use all three common tensor-product spherical grids discussed in Section \ref{sec:tensor_product} in the experiments, with $2m$ points in longitude and $n=m+1$ points in latitude, which gives the grid spacing in longitude as $h=\pi/m$.  We select the SLA time-step in the experiments according to a fixed multiple of the CFL number, which we set to $h/\|\mathbf{u}\|_{\infty}$, where $\|\mathbf{u}\|_{\infty}=2.93$.  For the cosine bells tests we use a time-step 10 times the CFL number, while for the Gaussian bells tests we just use the CFL number.  These values were chosen so that spatial errors dominate the computed solutions for most of the resolutions.  

The relative max-norm errors of the numerical solutions from the experiments are plotted in Figure \ref{fig:sla_errors} as a function of the grid resolution $m$.  For the cosine bells problem (see part (a) of the figure), we see that the errors are similar for all three grids and decrease as $\mathcal{O}(m^{-2})$ (or $\mathcal{O}(h^2)$).  This algebraic rate is expected since the initial condition is only $C^1(\Sph)$.  In contrast, we see that the errors decrease exponentially fast with $m$ for the $C^{\infty}(\Sph)$ Gaussian bells initial condition, until temporal errors begin to dominate the computed solutions after $m=128$.  

Finally, we compare the accuracy of the new barycentric SLA method to other methods from the literature for this same test problem in Table \ref{tbl:compare}.  We see for the cosine bells test case that the new method compares favorably to all the methods and gives a slightly smaller error than the RBF-PU based SLA method for the same resolution and time-step.  For the Gaussian bells problem, the new method clearly provides the best accuracy for comparable resolutions and time-steps to the other methods.  The next lowest error is produced by the Global RBF method, which is also spectrally accurate.  By doubling the time-step for the same resolution, the new method shows a further decrease in the errors indicating that temporal errors may still be dominating.  It should be noted, however, that the CSLAM, DG, and HOMME methods are all mass-conserving, whereas, the present SLA method is not.

%Plots of the relative errors and errors in mass at $t=5$ against the number of latitude points $n=m+1$  for the cosine bells experiment are shown in Figure \ref{fig:errcosinebells}. Similar results for the Gaussian initial conditions are shown in Figure \ref{fig:gaussianbells}.
\begin{figure}[t!]
	\centering	
	\begin{tabular}{cc}
		\includegraphics[width=0.45\textwidth]{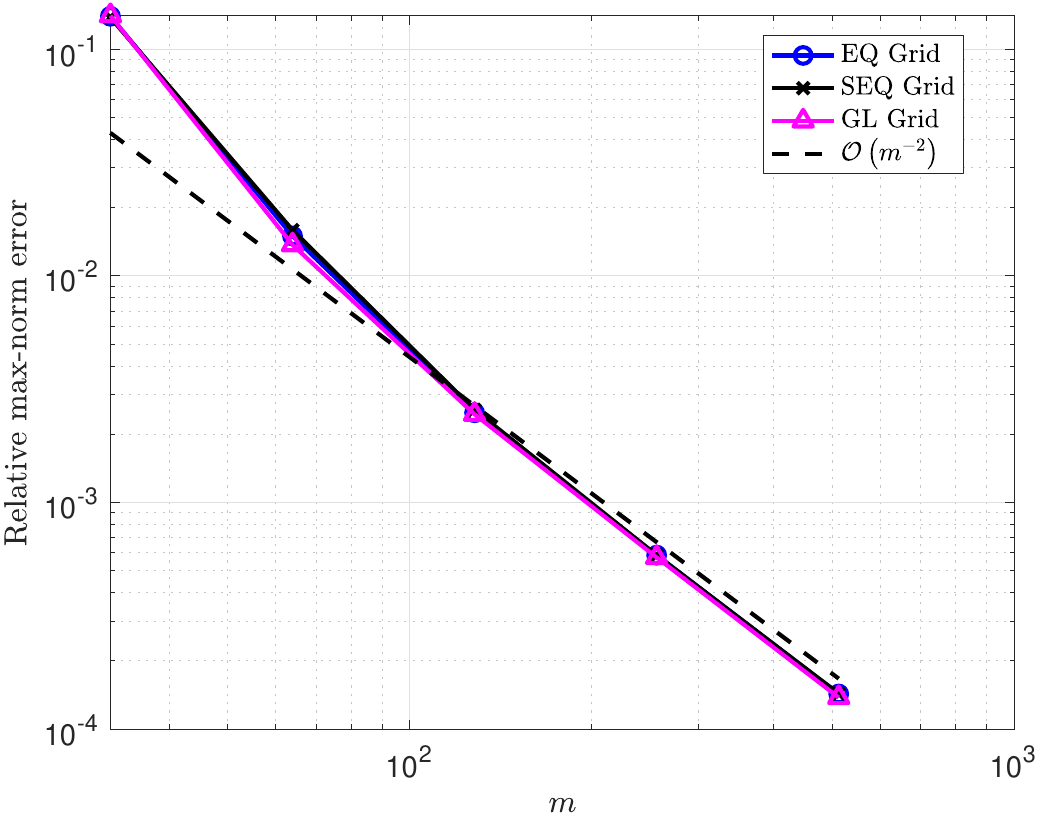}
		& 
		\includegraphics[width=0.46\linewidth]{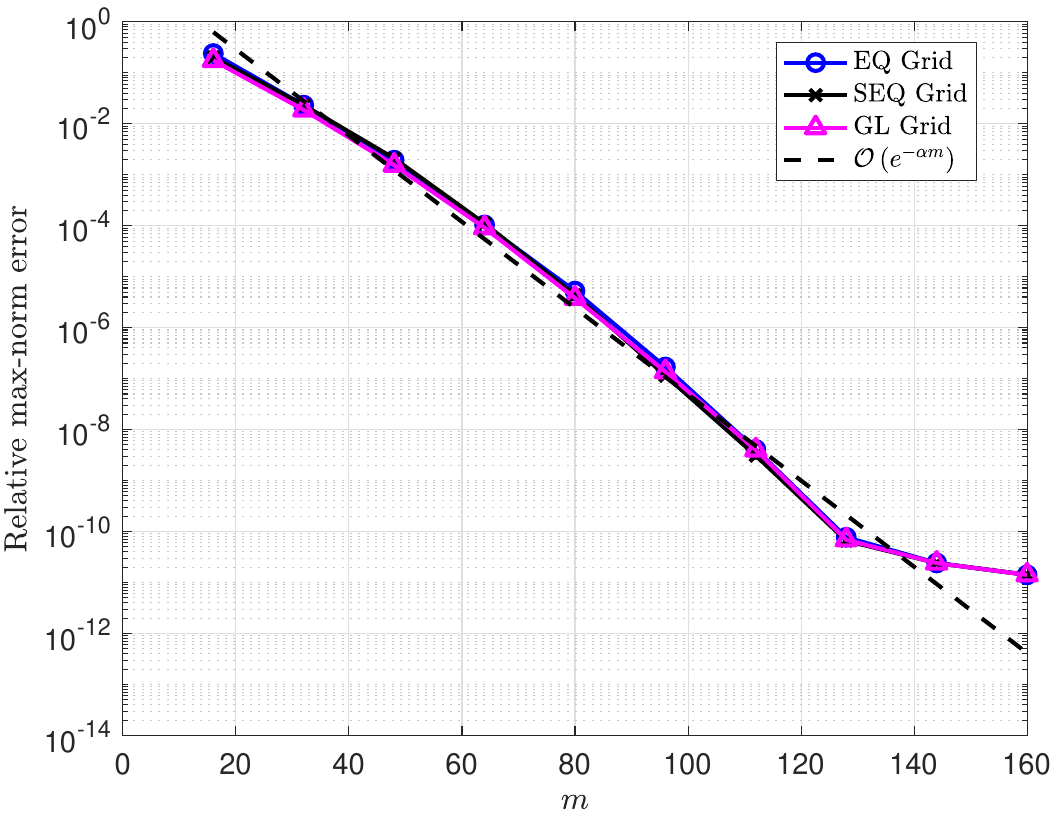}\\
		(a) & (b)
	\end{tabular}
	\caption{Spatial convergence results for (a) the cosine bells test case using a time-step equal to $10\times$CFL and (b) the Gaussian bells test case using a time-step equal to CFL.  All the grid types use $2m$ points in longitude and $n=m+1$ in latitude. Note (a) is plotted on a loglog scale while (b) is plotted on a semilogy scale and both plots display estimates of convergence rates as dashed lines.\label{fig:sla_errors}}
\end{figure}

\definecolor{Gray}{gray}{0.9}

\begin{table}[htb]
\centering
\begin{tabular}{|c||c|c|c|c|}
\hline
             & Resolution & Degrees of  & Number of & Relative $\ell_2$ \\
Method & (in degrees) & freedom & time-steps & error \\
\hline
\hline
\rowcolor{Gray}\multicolumn{5}{|c|}{Cosine Bells}\\
\hline
%CSLAM~\cite{NairLauritzen2010} & $1.5\degree$ & 21600 & 5/120 & $3.28 \times 10^{-2}$ \\
CSLAM~\cite{LauritzenEtAl2012} & $0.75\degree$ & 86400 & 240 & $\approx 6 \times 10^{-3}$ \\
DG, $p=3$~\cite{NairLauritzen2010} & $1.5\degree$ & 38400 & 2400 & $1.39 \times 10^{-2}$ \\
%RBF-FD, $n=84$~\cite{FoL11} & $1.5\degree$ & 23042 & 5/900 & $1.17 \times 10^{-2}$\\
%RBF-FD, $n=84$~\cite{FoL11} & $0.75\degree$ & 92162 & 5/2800 & $2.42 \times 10^{-3}$\\
%Local RBF, $n=84$ & $1.5\degree$ & 23042 & 5/35 & $3.45 \times 10^{-3}$\\
RBF-PU~\cite{shankar_2018}, $n=84$ & $1.5\degree$ & 23042 & 35 & $3.63 \times 10^{-3}$\\
Global RBF~\cite{shankar_2018} & $1.64\degree$ & 15129 &  45 & $5.10 \times 10^{-3}$\\
Barycentric DFS & $1.5\degree$ & 29040 & 35 & $3.25 \times 10^{-3}$ \\
%Global RBF 
\hline
\hline
\rowcolor{Gray}\multicolumn{5}{|c|}{Gaussian Bells}\\
\hline
%CSLAM~\cite{NairLauritzen2010} & $1.5\degree$ & 21600 & 5/120 & $2.46 \times 10^{-2}$ \\
CSLAM~\cite{LauritzenEtAl2012} & $0.75\degree$ & 86400 & 240 & $\approx 5 \times 10^{-4}$ \\
HOMME, $p=6$~\cite{LauritzenEtAl2014} & $1.5\degree$ & 29400 & 4800 & $\approx 3 \times 10^{-3}$ \\
%HOMME, $p=6$~\cite{LauritzenEtAl2014} & $0.75\degree$ & 117600 & 5/9600 & $6 \times 10^{-5}$ \\
%RBF-FD, $n=84$~\cite{FoL11} & $1.5\degree$ & 23042 & 5/900 & $3.18 \times 10^{-4}$\\
%RBF-FD, $n=84$~\cite{FoL11} & $0.75\degree$ & 92162 & 5/2800 & $2.42 \times 10^{-3}$\\
%Local RBF, $n=84$ & $1.5\degree$ & 23042 & 5/80 & $5.50 \times 10^{-5}$\\
RBF-PU~\cite{shankar_2018}, $n=84$ & $1.5\degree$ & 23042 & 80 & $1.35 \times 10^{-5}$\\
Global RBF~\cite{shankar_2018} & $1.64\degree$ & 15129 &  200 & $7.68 \times 10^{-8}$\\
Barycentric DFS & $1.5\degree$ & 29040 & 200 & $1.17 \times 10^{-8}$ \\
Barycentric DFS & $1.5\degree$ & 29040 & 400 & $7.99 \times 10^{-10}$ \\
\hline
\end{tabular}
\caption{Comparison of the barycentric DFS-based SLA scheme using the EQ grid with various methods from the literature for the two deformational flow test cases.  CSLAM refers to the conservative semi-Lagrangian multi-tracer transport scheme used in~\cite{LauritzenEtAl2012}, which uses a cubed-sphere grid. DG is the discontinuous Galerkin scheme from~\cite{NairLauritzen2010}, an Eulerian scheme employing $p=3$ degree polynomials (fourth-order accurate) on a cubed-sphere grid.  HOMME (High-Order Methods Modeling Environment) is an Eulerian scheme from~\cite{LauritzenEtAl2014}, also using a cubed-sphere grid, with results shown for a continuous Galerkin formulation using $p=6$ degree polynomials (seventh-order accurate).  RBF-PU (partition of unity) and global RBF refer to SLA advection schemes based on radial basis function interpolation from~\cite{shankar_2018}. Values marked with $\approx$ were estimated from plots of the relative $\ell_2$ errors in the referenced papers, as exact values were not provided.  The results given for the CSLAM, DG, and HOMME methods correspond to the non-filtered (or non shape-preserving) versions, which yielded the lowest errors for these test cases. Resolution (in degrees) indicates the approximate spacing of grid points (or nodes) around the equator.\label{tbl:compare}}
\end{table}

\section{Concluding remarks\label{sec:conclusion}}
We have introduced new bivariate barycentric formulas for the sphere and disk based on the DFS method.  These formulas are efficient in the sense that do not require a transform to trigonometric or polynomial coefficients, but instead work directly with the data and a set of weights that only depend on the grid. For many commonly used grids in applications, these weights are known explicitly or can be computed in a stable and efficient manner.  The formulas also deliver high-order accuracy for smooth solutions.  Moreover, the formulas are straightforward to implement, making them practical for a wide range of applications.  

The question of stability is more subtle, particularly given the results on the stability of trigonometric barycentric formulas in~\cite{austin2017numerical}.  However, as noted in that study, the trigonometric barycentric formula remains stable in most practical situations.  Whether this holds for the even/odd/$\pi$-periodic/$\pi$-anti-periodic trigonometric formulas used here remains an open question and warrants further investigation.  Nevertheless, our extensive numerical testing of the formulas, including their application to SLA, revealed no stability issues in practice.

The application of these new barycentric formulas to SLA for the tracer transport problem on the sphere shows promising results.  This technique compares favorably in terms of accuracy to existing methods, suggesting that further explorations on more complex atmospheric flows are justified.
 
Lastly, while this study focused on trigonometric and polynomial approximations for the sphere and disk, the resulting barycentric formulas could naturally extend to more general rational approximations.  For example, the even/odd polynomial formulas used for the disk could employ more general Floater-Hormann linear barycentric rational formulas~\cite{floater2007barycentric}, or even AAA type formulas~\cite{nakatsukasa2018aaa}.  This would only require a different technique for computing the weights.  Linear barycentric rational formulas have been used successfully for the disk and more general star-like domains in~\cite{berrut2021linear}.  Similarly, the trigonometric formulas derived here for the sphere and disk could be extended to rational trigonometric formulas using any of the approaches from~\cite{wilber2022data,bandiziol2019lebesgue,berrut2021bounding}.  In one dimension, rational barycentric formulas provide a richer approximation space, particularly for functions with singularities, and we expect that this benefit will also extend to the sphere and disk. 
%-----------------------------------------------------------------------------------------
	
\section*{Acknowledgements} 
The authors thank Anthony Austin for discussions about the $\pi$-antiperiodic barycentric interpolation formula discussed in Section \ref{sec:dfs_parity} and Alex Townsend and Daniel Fortunato for discussions about using DFS-based interpolation formulas within the SLA framework for solving PDEs on the sphere.  The authors work was partially supported by US National Science Foundation grants 1952674 and 2309712.

\section*{Conflict of Interest}

The authors declare no conflict of interest.

\bibliographystyle{siam}
\bibliography{compri}
	
\end{document}